\documentclass[11pt,reqno]{amsart}
\usepackage{amsthm,amssymb,amsmath,enumerate}
\usepackage{graphicx}
\usepackage{enumitem}

\newtheorem{theorem}{\sc Theorem}[section]
\newtheorem*{theorem*}{\sc Theorem}
\newtheorem{lemma}{\sc Lemma}[section]

\newtheorem*{corollary*}{\sc Corollary}
\newtheorem{proposition}{\sc Proposition}[section]

\newtheorem*{claim*}{\it Claim}

\theoremstyle{definition}

\newtheorem{definition}{\sc Definition}[section]
\newtheorem*{definition*}{\sc Definition}

\newtheorem{example}{\sc Example}[section]
\newtheorem{remark}{\bf Remark}[section]
\newtheorem*{remark*}{\sc Remark}
\newtheorem*{remarks}{\sc Remarks}

\newtheorem*{example*}{\sc Example}

\newcommand{\loc}{{\rm loc}}

\newcommand{\Real}{{\rm Re}}

\newcommand{\sprt}{{\rm sprt\,}}

\begin{document}

\title[Kolmogorov operator with vector field in Nash class]{Kolmogorov operator with the vector field in Nash class}

\author{D.\,Kinzebulatov}

\address{Universit\'{e} Laval, D\'{e}partement de math\'{e}matiques et de statistique, 1045 av.\,de la M\'{e}decine, Qu\'{e}bec, QC, G1V 0A6, Canada}

\email{damir.kinzebulatov@mat.ulaval.ca}

\thanks{
The research of D.K. is supported by grants from the Natural Sciences and Engineering Research Council of Canada (RGPIN-2017-05567) and Fonds de recherche du Québec - Nature et technologies (2019-NC-254946).
}

\author{Yu.\,A.\,Sem\"{e}nov}

\address{University of Toronto, Department of Mathematics, 40 St.\,George Str, Toronto, ON, M5S 2E4, Canada}

\email{semenov.yu.a@gmail.com}



\begin{abstract}
We consider divergence-form parabolic equation with measurable uniformly elliptic matrix and the vector field 
in a large class containing, in particular, the vector fields in $L^p$, $p>d$, 
as well as some vector fields that are not even in $L_{\loc}^{2+\varepsilon}$, $\varepsilon>0$. 
We establish H\"{o}lder continuity of the bounded soutions, sharp two-sided Gaussian bound on the heat kernel, Harnack inequality.

\end{abstract}

\maketitle

\section{Introduction}

A celebrated result of E.\,De Giorgi \cite{DG} and J.\,Nash \cite{N} states that the bounded solutions of the parabolic equation 
\begin{equation}
\label{eq0}
(\partial_t+A)u=0, \quad A=-\nabla \cdot a \cdot \nabla
\end{equation}
on $[0,\infty[ \times \mathbb R^d$, $d \geq 3$, 
with measurable matrix 
\begin{equation}
\label{H}
\tag{$H_{\sigma,\xi}$}
\begin{array}{c}
a=a^{\ast}:\mathbb R^d \rightarrow \mathbb R^d \otimes \mathbb R^d, \\
 \sigma I \leq a(x) \leq \xi I \quad \text{ for a.e. } x \in \mathbb R^d \quad \text{ for constants $0<\sigma<\xi<\infty$}
\end{array}
\end{equation}
are H\"{o}lder continuous, and the heat kernel $e^{-tA}(x,y)$ satisfies two-sided Gaussian bound with constants that depend only on $d$, $\sigma$, $\xi$. The purpose of this paper is to extend their result
to the
equation 
\begin{equation}
\label{op}
(\partial_t + \Lambda)u=0
\end{equation}
where
$$
\Lambda= -\nabla \cdot a \cdot \nabla + b \cdot \nabla
$$
with $b:\mathbb R^d \rightarrow \mathbb R^d$ in a large class of locally unbounded measurable vector fields. 

\medskip

\textbf{1.~}The existence and the precise form of the relationship between the integral characteristics of the coefficients $a$ and $b$ and the regularity properties of solutions to \eqref{eq0} and \eqref{op}
is one  of the classical and central problems in the theory of elliptic and parabolic PDEs.

By a result of D.\,G.\,Aronson \cite{A}, the heat kernel $e^{-t\Lambda}(x,y)$ of equation \eqref{op} satisfies two-sided Gaussian bound. By a result of S.\,D.\,Eidelman-F.\,O.\,Porper \cite{EP}, $t|\partial_te^{-t\Lambda}(x,y)|$ satisfies the  Gaussian upper bound. The constants in their bounds depend on $d$, $\sigma$, $\xi$, and the following integral characteristics of $b$:  $$\|b_1\|_p+\|b_2\|_\infty, \quad p>d$$ provided that $b_1+b_2=b$. 

Our first goal is to demonstrate, based on ideas of E.\,De Giorgi and J.\,Nash, that the constants in the two-sided bound on $e^{-t\Lambda}(x,y)$, in the upper bound on $t|\partial_te^{-t\Lambda}(x,y)|$, as well as H\"older continuity of bounded solutions to \eqref{op} (assuming first that the coefficients $a$, $b$ are smooth) depend in fact on a much finer characteristic of the vector field $b$, that is, on its \textit{elliptic Nash norm}:
\begin{equation*}
n_e(b,h):=  \sup_{x \in \mathbb{R}^d} \int_0^h \sqrt{e^{t\Delta}|b|^2(x) } \; \frac{dt}{\sqrt{t}} \quad (h>0),
\end{equation*}
and only on its elliptic Nash norm (Theorem \ref{apr_GB}).

Next, as is well known, the existence of even strong a priori estimates does not always mean that there is a satisfactory a posteriori regularity theory of the corresponding differential operator. 
Our second goal is to develop an exhaustive a posteriori theory of \eqref{op}, including two-sided Gaussian bound on the heat kernel of $-\nabla \cdot a \cdot \nabla + b \cdot \nabla$,  assuming only that $b$ is measurable, $|b| \in L^2_{\loc}$ and
$$\text{$n_e(b,h)$ is sufficiently small}$$
for some $h>0$ 
(Theorem \ref{nash_apost_thm}).

\begin{definition}
A measurable vector field $b:\mathbb R^d \rightarrow \mathbb R^d$ such that $|b| \in L^2_\loc$ is said to be in the Nash class $\mathbf{N}_e$ if $$n_e(b,h)<\infty$$
for some $h>0$.
\end{definition}

The class $\mathbf{N}_e$ contains the vector fields $b=b_1+b_2$ with $\|b_1\|_p+\|b_2\|_\infty<\infty$, $p>d$. For such $b$ one has $\lim_{h\downarrow 0} n_e(b,h)= 0$. The class $\mathbf{N}_e$ also contains some vector fields $b$ with $|b|$ not even in $L^{2+\varepsilon}_{\loc}$, $\varepsilon>0$. See more detailed discussion in Section \ref{main_sect}.
 The elliptic Nash norm $n_e(b,h)$ was introduced in \cite{S2} where the two-sided Gaussian bound on the heat kernel $e^{-t\Lambda}(x,y)$ was obtained under some additional to $b \in \mathbf{N}_e$ assumptions.

If $a=I$ or $a$ is H\"{o}lder continuous, then the condition $|b| \in L^1_\loc$ and
$$
\text{$\kappa_{d+1}(b,h)$ is sufficiently small}
$$
for some $h>0$,
where
 \[
\kappa_{d+1}(b,h):=\sup_{x\in\mathbb{R}^d}\int_0^h e^{t\Delta}|b|(x)\frac{dt}{\sqrt{t}} \qquad (\text{Kato norm of $b$}),
\]
provides 
the upper Gaussian bound \cite{S}, the Harnack inequality and the lower Gaussian bound on the heat kernel $e^{-t\Lambda}(x,y)$ \cite{Z1}, see also \cite{Z2}. The class of the vector fields $b$ such that $|b| \in L^1_\loc$ and $$\kappa_{d+1}(b,h)<\infty$$ for some $h>0$ is the well known Kato class  $\mathbf{K}^{d+1}$. (The results in \cite{Z1,Z2} were obtained, in fact, for $b=b(t,x)$ in the non-autonomous Kato class, itself introduced by Q.\,S.\,Zhang.)

Thus, the Nash class $\mathbf{N}_e$ is an analogue of the Kato class $\mathbf{K}^{d+1}$ in case $a=a(x)$ is only measurable. Note that $\mathbf{N}_e \subset \mathbf{K}^{d+1}$ as is immediate from elementary inequality $e^{t\Delta}|b|(x) \leq \sqrt{e^{t\Delta}|b|^2(x)}$.

The principal difference between the cases covered by the Nash class $\mathbf{N}_e$ ($a$ is measurable) and the Kato class $\mathbf{K}^{d+1}$ ($a$ is 
H\"{o}lder continuous) is as follows. For H\"{o}lder continuous $a$ one can appeal, in the proof of the two-sided bound, to the estimate $|\nabla_x e^{-tA}(x,y)| \leq C t^{-\frac{1}{2}}e^{ct\Delta}(x,y)$, which does not hold for merely measurable $a$; for such $a$ the role of the previous estimate is assumed by far-reaching inequalities  $$\mathcal{N}(t)\leq \frac{c_0}{t}, \quad \hat{\mathcal{N}}(t)\leq \frac{\hat{c}_0}{t},$$ where $\mathcal{N}(t)$, $\hat{\mathcal{N}}(t)$ are the so-called Nash's functions similar to  $$\langle\nabla_x p\cdot\frac{a(x)}{p}\cdot\nabla_x p\rangle, \quad p \equiv p(t,x,y)=e^{-tA}(x,y)$$ employed by J.\,Nash in \cite{N}. See Sections \ref{app_N} and \ref{apr_GB_sect} for details.
 
We comment more on the relationship between the Nash class and the Kato class in Section \ref{sect3} below.

\medskip

\textbf{2.~}In the context of the semigroup theory of \eqref{op}, the standard assumption on the vector field $b$ used in the literature is the form-boundedness condition: there exist constants $\delta>0$ and $c(\delta) \geq 0$ such that the quadratic inequality
\begin{equation*}
 \|\sqrt{b \cdot a^{-1} \cdot b}\, f \|_2^2\leq \delta\|A^\frac{1}{2}f\|_2^2+c(\delta)\|f\|_2^2,
\end{equation*}
holds for all $f \in W^{1,2}$.
Briefly,
$$b\cdot a^{-1}\cdot b \leq \delta A+c(\delta) \quad \text{ (in the sense of quadratic forms)}
$$ 
(written as $b \in \mathbf{F}_\delta(A)$).
This is a large class of singular vector fields containing e.g.\,the vector fields $b=b_1+b_2$ with $|b_1|$ in $L^d$ or in the weak $L^d$ class, $|b_2| \in L^\infty$,
see discussion below (before Theorem \ref{grad_thm}).

If $b \in \mathbf{F}_\delta(A)$ with $\delta<1$, then the corresponding to 
$\Lambda=-\nabla \cdot a \cdot \nabla + b \cdot \nabla$ quadratic form on $W^{1,2}$ is quasi $m$-accretive, 
and so it determines an operator $\Lambda_2$ in $L^2$ generating a holomorphic semigroup. 
The equation \eqref{op} with $\Lambda=\Lambda_2$ possesses a detailed regularity theory in $L^2$ and, moreover, in $L^p$, $p>\frac{2}{2-\sqrt{\delta}}$, but not in $L^1$. See  Section \ref{sect3} for more details.

If $b \in \mathbf{N}_e$, then the situation is different: the equation \eqref{op} 
does not seem to admit any $L^p$ theory for $p>1$ beyond the existence of a semigroup. However, it 
admits a detailed $L^1$ theory. In Theorem \ref{nash_apost_thm} we 
construct an operator realization $\Lambda_1$ of the formal operator $\Lambda$ in $L^1$ as the 
\textit{algebraic sum}
$$
\Lambda_1=A_1 + (b \cdot \nabla)_1, \quad D(\Lambda_1)=D(A_1),
$$
where $A_1$ is the operator realization of $-\nabla \cdot a \cdot \nabla$ in $L^1$ and $(b \cdot \nabla)_1$ is the closure of $b \cdot \nabla$ in the graph norm of $A_1$, and show that 
$$
e^{-t\Lambda_1}=s\mbox{-}L^1\mbox{-}\lim_{\varepsilon \downarrow 0}e^{-t\Lambda_1^\varepsilon} \quad \text{(\loc.\,uniformly in $t \geq 0$)}
$$
where 
$
\Lambda_1^\varepsilon=-\nabla \cdot a_\varepsilon \cdot \nabla + b_\varepsilon \cdot \nabla$ of domain $ D(\Lambda_1^\varepsilon)=(1-\Delta)^{-1}L^1$ with smooth ($a_\varepsilon$, $b_\varepsilon$) approximating $(a,b)$ and
essentially non-increasing the Nash norm: $$n_e(b_\varepsilon,h) \leq n_e(b,h)+\tilde{c}\varepsilon.$$

Armed with the last results and a priori two-sided Gaussian bound on $e^{-t\Lambda^\varepsilon}(x,y)$ of Theorem \ref{apr_GB}, we develop an exhaustive regularity theory of \eqref{op}, including a posteriori two-sided Gaussian bound on the heat kernel $e^{-t\Lambda}(x,y)$, the Harnack inequality, the H\"{o}lder continuity of bounded solutions of \eqref{op}, the strong Feller property, and the Gaussian upper bound on $t|\partial_te^{-t\Lambda}(x,y)|$ with the optimal (up to a strict inequality) exponent in the Gaussian factor. We also establish the bounds $$\|\nabla(\mu+\Lambda_1)^{-\alpha}\|_{1\rightarrow 1}\leq C\mu^{-\frac{2\alpha-1}{2}}$$ for $\frac{1}{2}<\alpha\leq 1$, $\mu>\mu_0>0$ ($\mu_0$ depends on $d, \sigma,\xi, n_e(b,h)$), and $$\|\nabla e^{-t\Lambda_1}\|_{1 \rightarrow 1} \leq ct^{-\frac{1}{2}}e^{\omega t}, \quad t>0,$$ see Theorem \ref{grad_thm}.

We conclude this introduction by mentioning that the condition $b \in \mathbf{F}_\delta(A)$, $\delta<\infty$ provides two-sided Gaussian bounds on the heat kernel of $-\nabla \cdot a \cdot \nabla + b \cdot \nabla$ but only as long as ${\rm div\,}b$ satisfies additional integral constraints (that is, ${\rm div\,}b$ is in the Kato class $\mathbf{K}^d$, cf.\,Section \ref{sect3}), see \cite{KiS5}.

\bigskip

\tableofcontents

\section{Preliminaries}

We will need the following standard notations and results.

\medskip

\textbf{1.~}Let $\mathcal B(X,Y)$ denote the space of bounded linear operators between Banach spaces $X \rightarrow Y$, endowed with the operator norm $\|\cdot\|_{X \rightarrow Y}$. $\mathcal B(X):=\mathcal B(X,X)$.

We write $T=s\mbox{-} X \mbox{-}\lim_n T_n$ for $T$, $T_n \in \mathcal B(X,Y)$ if $$\lim_n\|Tf- T_nf\|_Y=0 \quad \text{ for every $f \in X$}.
$$ 

Denote by $[L^p]^d$ and $[L^p]^{d \times d}$ the spaces of the $d$-vectors  and the $d \times d$-matrices with entries in $L^p \equiv L^p(\mathbb R^d,dx)$.

Put
$$
\langle f,g\rangle = \langle f\bar{g}\rangle :=\int_{\mathbb R^d}f\bar{g}dx$$ 
and
$
\|\cdot\|_{p \rightarrow q}=\|\cdot\|_{L^p \rightarrow L^q}.
$

$C_{\infty}:=\{f \in C(\mathbb R^d)\mid \lim_{|x| \rightarrow \infty}f(x)=0\}$ endowed with the $\sup$-norm.

$\mathcal W^{\alpha,1}$, $\alpha>0$, is the Bessel potential space endowed with norm $\|u\|_{1,\alpha}:=\|g\|_1$,  
$u=(1-\Delta)^{-\frac{\alpha}{2}}g$, $g \in L^1$.

Let $E_\varepsilon f := e^{\varepsilon\Delta} f$ ($\varepsilon>0$), the De Giorgi mollifier of $f$.

For a vector field $b$ we put $b^2:=|b|^2$ and $b_a^2:=b \cdot a^{-1} \cdot b$.

We write $c \neq c(\varepsilon)$ to emphasize that $c$ is independent of $\varepsilon$.

Put
\begin{align*}
k_\mu(t,x,y)  \equiv k(\mu t,x,y) :=(4\pi\mu t)^{-\frac{d}{2}}e^{ - \frac{|x-y|^2}{4\mu t}}, \quad \mu>0.
\end{align*}

\medskip

\textbf{2.~}Let $a \in (H_{\sigma,\xi})$, $0<\sigma<\xi<\infty$. Let $p(t,x,y)$ be the heat kernel of $-\nabla \cdot a \cdot \nabla$ (that is, $p(t,x,y)=e^{-tA}(x,y)$ in the notation of the next section).

\begin{theorem}
\label{thm_p}
Fix constants $0<c_2<\sigma$ and $c_4>\xi$. There exist constants $c_1$, $c_3>0$ that depend only on $d,c_2,c_4$ such that, for all $t>0$, $x,y \in \mathbb{R}^d$,
\[
p(t,x,y) \leq c_3 k_{c_4} (t,x-y)
\tag{${\rm UGB}^p$}
\label{UGB_p}
\]
and
\[
c_1 k_{c_2} (t,x-y) \leq p(t,x,y).
\tag{${\rm LGB}^p$}
\label{LGB_p}
\]
Also, for a given $c_6 > \xi$ there is a generic constant $c_5$ depending on $c_6$ such that 
\[
	t |\partial_t p(t,x,y)| \leq c_5 k_{c_6} (t, x-y) 
\label{GB_pt}
\tag{${\rm UGB}^{\partial_tp}$}	
\]
for all $t>0$, $x, y \in \mathbb{R}^d$.
\end{theorem}

The proof of \eqref{UGB_p} and \eqref{LGB_p} with \textit{some} constants $c_2$ and $c_4$ is due to \cite{A}.  The proof of \eqref{GB_pt} with some constant $c_6$ is due to \cite{EP}. 
The proof of \eqref{UGB_p} and \eqref{GB_pt} in the form as stated is due to \cite{KS}, and in a strengthened form, i.e.\,with polynomial factor,
can be found in \cite{Da}. The proof of  \eqref{LGB_p} as stated is due to \cite{S}.

\medskip

\textbf{3.~}Recall that if $S$ and $T$ are linear operators in a Banach space $(Y,\|\cdot\|)$, 
then $S$ is said to be $T$-bounded if $D(S) \supset D(T)$ and there exist constants $\eta$ and $c$ such that 
$$\|Sy\| \leq \eta\|Ty\|+c\|y\| \quad \text{for all $y \in D(T)$}.
$$ 

By $T \upharpoonright X$ we denote the restriction of $T$ to a subset $X \subset D(T)$. 

By 
$
\big(T \upharpoonright X\big)_{Y \rightarrow Y}^{\rm clos}
$
we denote the closure of $T \upharpoonright X$ (when it exists).

 Next, let operator $T$ be closed. A subset $D_T\subset D(T)$ is called a core of $T$ if $$(T\upharpoonright D_T)_{Y \rightarrow Y}^{\rm clos}=T.$$

Let $P$,  $Q$ be linear operators in a Banach space $Y$. Assume that $Q$ is closed, $D(P)$ contains a core $D_Q$ of $Q$ and $\|Py\|\leq \eta\|Qy\|+c\|y\|$, $y\in D_Q$ ($\eta, c$ some constants). This inequality extends by continuity to $D(Q)$. An extension of $P$ obtained in this way, say $\tilde{P}$, is $Q$-bounded.

\bigskip

\section{Main results}

\label{main_sect}

\textbf{1.~}We first prove \textit{a priori} Gaussian lower and upper bounds on the heat kernel of $-\nabla \cdot a \cdot \nabla + b \cdot \nabla$, $a \in (H_{\sigma,\xi})$. In what follows, $d \geq 3$.

\begin{definition}
We say that a constant is generic if it depends only on the dimension $d$ and the constants $\sigma$ and $\xi$. 
\end{definition}

\begin{theorem}
\label{apr_GB}
Let $a \in (H_{\sigma,\xi})$ be smooth, let $b:\mathbb R^d \rightarrow \mathbb R^d$ be smooth and bounded, $\xi_1>\xi$. There exists a generic constant $\tilde{n}>0$ such that if the Nash norm of $b$ $$n_e(b,h) \equiv \sup_{x \in \mathbb{R}^d} \int_0^h \sqrt{e^{t\Delta}|b|^2(x) } \; \frac{dt}{\sqrt{t}}$$
satisfies
$$
n_e(b,h) \leq \tilde{n}
$$
 for some $h>0$, then there exist positive constants $\sigma_1<\sigma$ and $c_{\sigma_1}$, $c_{\xi_1}>0$, $\omega_i \geq 0,$ $i=1,2,$ such that the heat kernel $u(t,x,y)$ of $-\nabla \cdot a \cdot \nabla + b \cdot \nabla$ satisfies the Gaussian lower and upper bounds 
\[
c_{\sigma_1} e^{-t\omega_1} k_{\sigma_1} (t,x-y) \leq u(t,x,y) \leq c_{\xi_1}  e^{t\omega_2} k_{\xi_1} (t,x-y)
\tag{${\rm LUGB}^u$}
\label{GB_u}
\]
for all $t>0$ and $x,y \in \mathbb{R}^d$.
The constants $\sigma_1$, $c_{\sigma_1}, c_{\xi_1}, \omega_i$  depend only on $d,\xi_1$ and $n_e(b,h)$.
\end{theorem}

\begin{definition}
We say that a constant is generic* if it depends on $d$, $\sigma$, $\xi$ and on the Nash norm $n_e(b,h)$ of the vector field $b$. 
\end{definition}

Thus, the constants in \eqref{GB_u} are generic*.
The fact that they do not depend on the smoothness of $a$, $b$, coupled with the next Proposition \ref{prop1} and a careful approximation argument, will allow us to establish the corresponding a posteriori heat kernel bounds (Theorem \ref{nash_apost_thm}).

\medskip

\textbf{2.~}Recall that a vector field $b \in [L^2_{\loc}]^d$ is said to be in the Nash class $\mathbf{N}_e$ if $$n_e(b,h)<\infty$$
for some $h>0$.

\begin{example}

(1) We have $$\text{$|b| \in L^p$, $p>d$ \quad $\Rightarrow$ \quad $b \in \mathbf{N}_e$},$$
as follows easily using $\|e^{t\Delta}\|_{r \rightarrow \infty} \leq Ct^{-\frac{d}{2r}}$ upon taking $r=\frac{p}{2}$:
\begin{align*}
\sup_{x \in \mathbb R^d}\int_0^h\sqrt{ e^{t\Delta}|b|^2(x)} \frac{dt}{\sqrt{t}} & \leq \int_0^h\sqrt{\|e^{t\Delta}|b|^2\|_\infty}\frac{dt}{\sqrt{t}} \\
& \leq C^{\frac{1}{2}}\int_0^h\sqrt{t^{-\frac{d}{p}}\|b\|^2_p}\frac{dt}{\sqrt{t}} \\
& =C^{\frac{1}{2}}\frac{2p}{p-d}h^\frac{p-d}{2p}\|b\|_p<\infty.
\end{align*}

\smallskip

(2) There exist $b \in \mathbf{N}_e$ such that, for any $\varepsilon>0$, $|b| \not\in L_{\loc}^{2+\varepsilon}$, e.g.\,consider   
$$
|b(x)|=\mathbf{1}_{B(0, e^{-1})}(x)|x_1|^{-\frac{1}{2}}|\log |x_1||^{-\alpha}, \quad \alpha>\frac{1}{2}, 
$$
where $x=(x_1,\dots,x_d)$.
\end{example}

\medskip

\textbf{3.~}Let $A \equiv A_2$ be the self-adjoint operator in $L^2$ associated with the quadratic form  $\langle \nabla u, a \cdot \nabla u\rangle$, $u \in W^{1,2}$.
A standard application of the Beurling-Deny theory yields that the operator $A$ generates a symmetric Markov semigroup $e^{-tA}$.
Then
$$
e^{-tA_1}:=\biggl[e^{-tA} \upharpoonright L^1 \cap L^2 \biggr]_{L^1 \rightarrow L^1}^{\rm clos} \in \mathcal B(L^1), \quad t>0.
$$
is a $C_0$ semigroup (this is a general fact from the theory of symmetric Markov semigroups).
Its generator $-A_1$ is an appropriate operator realization of the formal operator $-\nabla \cdot a \cdot \nabla$ in $L^1$.

Given a vector field $b\in [L^1_{\loc}]^d$, we define in $L^1$ operator $B_{\max} \supset b\cdot \nabla$ of domain $$D(B_{\max}):=\{f\in L^1 \mid f \in W^{1,1}_{\loc} \text{ and } b\cdot\nabla f\in L^1\}.$$

The following result will allow us to construct an operator realization of the formal Kolmogorov operator $-\nabla \cdot a \cdot \nabla + b \cdot \nabla$, with $a \in (H_{\sigma,\xi})$ measurable and $b \in \mathbf{N}_e$ locally unbounded, in $L^1$.

\medskip

\begin{proposition} \label{prop1}

Let $b \in \mathbf{N}_e$. Then $D(B_{\max})\supset D(A)\cap D(A_1)$ and $B_{\max}\upharpoonright D(A_1) \cap D(A)$
extends by continuity in the graph norm  of $A_1$ to  $A_1$-bounded operator $(b \cdot \nabla)_1$:
$$
\|(b \cdot \nabla)_1 f\|_1 \leq \eta\|A_1 f\|_1+\eta \mu \|f\|_1, \quad f \in D(A_1),
$$
with bound $\eta:=\frac{1}{1-e^{-\mu h}}\sqrt{\frac{c_0}{\sigma c_4}}\;n_e(b,h c_4)$, $\mu > 0$. Here and below,
$$c_0:=2 c_3c_5 + \frac{d}{2},$$ where $c_i$ ($i=3,4,5$) are generic constants in the Gaussian bounds on the heat kernel $e^{-tA}(x,y)$ and its time derivative in Theorem \ref{thm_p}.
\end{proposition}

We will also need the following standard result. Since $e^{-tA_1}$ and $e^{-tA}$ have the same integral kernel $e^{-tA}(x,y)$ which satisfies $|\partial_t e^{-tA}(x,y)| \leq c_5 t^{-1}k_{c_6} (t, x-y)$ (Theorem \ref{thm_p}),
 there exists a generic constant $C>0$ such that $(CtD_te^{-tA_1})^n$  are uniformly (in $0 \leq t \leq 1$ and $n=1,2\dots$) bounded in $\mathcal B(L^1)$, and so, by a classical result \cite[Ch.\,IX, sect.\,10]{Y}, 
\begin{equation}
\label{__A_1_bd}
\|(\zeta + A_1)^{-1}\|_{1 \rightarrow 1} \leq \frac{M}{|\zeta|}, \quad \Real \zeta>0
\end{equation}
with generic constant $M$.

\begin{theorem}
\label{nash_apost_thm}
Let $a \in (H_{\sigma,\xi})$, $b \in \mathbf{N}_e$ with the Nash norm $$n_e(b,hc_4) <\sqrt{\frac{\sigma c_4}{c_0}}$$ for some $h>0$ (the constants $c_0$, $c_4$ were introduced above).

The following is true:

\smallskip

{\rm (\textit{i})} The algebraic sum $\Lambda_1:=A_1 + (b \cdot \nabla)_1$, $D(\Lambda_1)=D(A_1)$ generates a quasi bounded holomorphic semigroup $e^{-t\Lambda_1}$ in $L^1$ with the sector of holomorphy $$\{z \in \mathbb C \mid |{\rm arg\,} z|<\frac{\pi}{2} - \theta\}, \quad \text{ where } \tan \theta=\sqrt{2}\biggl(\frac{M}{1-\sqrt{\frac{c_0}{\sigma c_4}}n_e(b,h c_4)}-1\bigg).$$
The operator $\Lambda_1$ is an operator realization of the formal Kolmogorov operator $-\nabla \cdot a \cdot \nabla + b \cdot \nabla$ in $L^1$.

\smallskip

{\rm (\textit{ii})}
$$
e^{-t\Lambda_1} =s{\mbox-}L^1{\mbox-}\lim_{\varepsilon \downarrow 0}e^{-t\Lambda_1^\varepsilon} \quad (\text{loc.\,uniformly in $t \geq 0$}),
 $$
where 
$$\Lambda_1^\varepsilon:=-\nabla \cdot a_\varepsilon \cdot \nabla + b_\varepsilon \cdot \nabla, \quad  D(\Lambda_1^\varepsilon)=\mathcal W^{2,1}$$ 
are the approximating operators, with smooth matrices $a_\varepsilon \in (H_{\sigma,\xi})$ and  smooth bounded vector fields $b_\varepsilon$ constructed in such a way that
 $$a_\varepsilon \rightarrow a \quad \text{ strongly in } [L^2_{\loc}]^{d \times d}, \quad b_\varepsilon\rightarrow b \quad \text{ strongly in } [L^2_\loc]^d \quad \text{ as } \varepsilon \downarrow 0,$$ 
and the Nash norm of $b_\varepsilon$ for all small $\varepsilon>0$ is controlled by the Nash norm of $b$: 
$$n_e(b_\varepsilon,h) \leq n_e(b,h) + \tilde{c}\varepsilon \quad (\tilde{c}\text{ generic constant}).$$

The semigroup $e^{-t\Lambda_1}$ conserves positivity and is a $L^\infty$ contraction (and so the convergence in {\rm(\textit{ii})} holds for $e^{-t\Lambda_r}$ in $L^r$ for all $1<r<\infty$).

\medskip

Moreover, there exists a generic constant $\tilde{n}>0$ such that if 
$n_e(b,h c_4) \leq \tilde{n}$, then we further have:

 \smallskip

{\rm (\textit{iii})} For every $t>0$, $e^{-t\Lambda_1}$ is an integral operator.

\smallskip

{\rm (\textit{iv})} The heat kernel $e^{-t\Lambda}(x,y)$ ($\equiv$ the integral kernel of $e^{-t\Lambda_1}$) satisfies, possibly after redefinition on a measure zero set in $\mathbb R^d \times \mathbb R^d$, the lower and upper Gaussian bounds: 

For every $\xi_1>\xi$ there exist generic* constants $\sigma_1\in ]0,\sigma[$ and $c_i >0$, $\omega_i \geq 0$, $i=1,2$ such that
\[
c_1 e^{-t \omega_1} k_{\sigma_1} (t,x-y) \leq e^{-t\Lambda}(x,y) \leq c_2 e^{t\omega_2} k_{\xi_1} (t,x-y)
\]
for all $t>0$, $x, y \in \mathbb{R}^d$.

\smallskip

{\rm (\textit{v})} $e^{-t\Lambda_1}$ conserves probability:
$$
\langle e^{-t\Lambda}(x,\cdot)\rangle=1 \quad \text{ for every } x \in \mathbb R^d.
$$

\smallskip

{\rm (\textit{vi})} For every $f \in L^1$, $u(t,\cdot):=e^{-t\Lambda_1}f(\cdot)$ is H\"{o}lder continuous  (possibly after redefinition on a measure zero set in $\mathbb R^d \times \mathbb R^d$), i.e.\,for every $0<\alpha<1$ there exist generic* constants $C<\infty$ and $\beta \in ]0,1[$ such that for all $z \in \mathbb R^d$, $s>R^2$, $0<R \leq 1$
$$
|u(t,x)-u(t',x')| \leq C
\|u\|_{L^\infty([s-R^2,s] \times \bar{B}(z,R))}
\biggl(\frac{|t-t'|^{\frac{1}{2}} + |x-x'|}{R} \biggr)^\beta
$$
for all $(t,x)$, $(t',x') \in [s-(1-\alpha^2)R^2,s] \times \bar{B}(z,(1-\alpha) R)$.

Furthermore, $u \geq 0$ satisfies the Harnack inequality: Let $0<\alpha<\beta<1$ and $\gamma \in ]0,1[$, then there exists a constant $K=K(d,\sigma,\xi,\alpha,\beta,\gamma)<\infty$ such that for all $(s,x) \in ]R^2,\infty[ \times \mathbb R^d$, $0<R \leq 1$ one has
$$
u(t,y) \leq K u(s,x)
$$
for all $(t,y) \in [s-\beta R^2,s-\alpha^2 R^2] \times \bar{B}(x,\delta R)$.

\smallskip

{\rm(\textit{vii})} 
$$e^{-t\Lambda_{C_\infty}}:=\bigl[e^{-t\Lambda_1} \upharpoonright C_\infty \cap L^1\bigr]^{\rm clos}_{C_\infty \rightarrow C_\infty}, \quad t>0$$ is a Feller semigroup in $C_\infty$ having the property
$
e^{-t\Lambda_{C_\infty}}[L^\infty \cap L^1] \subset C_\infty,
$ $t>0$. Moreover, 
$$e^{-t\Lambda_{C_u}}f(x):=\langle e^{-t\Lambda}(x,\cdot)f(\cdot)\rangle, \quad t>0$$ is a Feller semigroup on $C_u$, the space of bounded uniformly continuous functions on $\mathbb R^d$.

\smallskip

{\rm(\textit{viii})} For every $c_6>\xi$ there exists a generic* constant $c_5$  such that
$$
|\partial_t e^{-t(\omega_2+\Lambda_1)}(x,y)| \leq c_5 t^{-1} k_{c_6}(t,x-y)
$$
for all $t>0$, $x$, $y \in \mathbb R^d$.

{\rm(\textit{ix})} For every $1<p<\infty$,
$$
e^{-t\Lambda_p}:=\biggl[e^{-t\Lambda_1} \upharpoonright L^1 \cap L^p \biggr]^{\rm clos}_{L^p \rightarrow L^p}
$$
is a quasi bounded holomorphic semigroup with the same sector of holomorphy as in {\rm(\textit{i})}. 

\smallskip

{\rm(\textit{x})} For every $\frac{1}{2}<\alpha \leq 1$,
$$\|\nabla (\zeta+\Lambda_1)^{-\alpha}\|_{1 \rightarrow 1} \leq C(\Real \zeta)^{-\alpha+\frac{1}{2}}.$$

\end{theorem}

\bigskip

\textbf{4.~}Recall that a vector field $b$ is said to be form-bounded (with respect to $A \equiv A_2$) if
there exist finite constants $\delta>0$ and $c(\delta) \geq 0$ such that the quadratic inequality
\begin{equation*}
 \|b_a f \|_2^2\leq \delta\|A^\frac{1}{2}f\|_2^2+c(\delta)\|f\|_2^2 
\end{equation*}
is valid for all $f \in D(A^\frac{1}{2}) \equiv W^{1,2}$, where $b_a:=\sqrt{b\cdot a^{-1}\cdot b}$. We write $b \in \mathbf{F}_\delta(A)$. 

It is easily seen that
$$
b \in \mathbf{F}_\delta(-\Delta) \quad \Rightarrow \quad b \in \mathbf{F}_{\delta_a}(A) \text{ with } \delta_a=\sigma^{-2}\delta.
$$
The class $\mathbf{F}_\delta(A)$  contains, in particular, the vector fields $$b=b_1+b_2, \quad |b_1| \in L^d, \quad |b_2| \in L^\infty,$$
and for every such $b$ the form-bound $\delta$ can be chosen arbitrarily small. The class $\mathbf{F}_\delta(A)$ also contains vector fields having critical-order singularities. For instance,
$$
b(x)=\pm \sqrt{\delta} \frac{d-2}{2}|x|^{-2}x \in \mathbf{F}_\delta(-\Delta) \quad \text{ with } c(\delta)=0
$$
(by Hardy's inequality). More generally, $\mathbf{F}_\delta(A)$ contains the vector fields $b=b_1+b_2$ with $|b_1|$ in  the weak $L^d$ class or the Campanato-Morrey class, and $|b_2| \in L^\infty$, with $\delta$ depending on the norm of $|b_1|$ in the respective classes.  Moreover, for every $\varepsilon>0$ one can find vector fields $b \in \mathbf{F}_\delta(A)$ such that $|b| \not \in L^{2+\varepsilon}_{\loc}$. We refer to \cite[sect.\,4]{KiS} for details and other examples.

\begin{theorem}
\label{grad_thm}
Let $d \geq 3$, assume that $b \in \mathbf{N}_e$ with the same norm $n_e(b,h)$ as in Theorem \ref{nash_apost_thm}{\rm(\textit{iii})-(\textit{x})} for some $h>0$. Additionally, assume that $ b \in \mathbf{F}_\beta(-\Delta)$ for some $\beta<\infty$.
Then
 \begin{equation}
\label{grad_est}
\|\nabla e^{-t\Lambda_1}\|_{1 \rightarrow 1} \leq C t^{-\frac{1}{2}} e^{\omega_2 t}, \quad t>0,
\end{equation}
with constant $C$ depending on $d$, $\sigma$, $\xi$, $n_e(b,h)$, $\beta$ and $c(\beta)$.
\end{theorem}

\bigskip

\begin{remark}
It is not clear how to extend \eqref{grad_est} and the bound in Theorem \ref{nash_apost_thm}(\textit{x}) to
\begin{equation}
\label{p_bd} 
\tag{$\ast$}
\|\nabla e^{-t\Lambda_p}\|_{p \rightarrow p} \leq C_p t^{-\frac{1}{2}}e^{\nu_p t}, \quad \|\nabla (\zeta + \Lambda_p)^{-1}\|_{p \rightarrow p} \leq c_p (\Real \zeta)^{-\frac{1}{2}}
\end{equation}
 for \textit{some} $p>1$. Of course, if also $b \in \mathbf{F}_\beta(A)$ with $\beta<1$, then by standard theory $\|\nabla e^{-t\Lambda_2}\|_{2 \rightarrow 2} \leq C_2t^{-\frac{1}{2}}e^{\nu_2 t}$, $t>0$ for constants $C_2$, $\nu_2$ depending on $d$, $\xi$, $\sigma$, $\beta$ and $c(\beta)$, and so \eqref{p_bd} follows by interpolation for all $p \in [1,2]$ (similarly for $\nabla (\zeta + \Lambda_p)^{-1}$).
\end{remark}

\bigskip

\section{Nash's function $\mathcal N_\delta$}

\label{app_N}

Put $p(t,x,y) \equiv p_\varepsilon(t,x,y):=e^{-tA^\varepsilon}(x,y),$ where $A^\varepsilon:=-\nabla \cdot a_\varepsilon \cdot \nabla$, $a_\varepsilon \equiv E_\varepsilon a$ (the De Giorgi mollifier, see above). 
Below we write for brevity 
$a \equiv a_\varepsilon$.

Define Nash's function
\[
	\mathcal N_\delta (t,x) := \big\langle \nabla_\cdot p(t,\cdot,x) \cdot \frac{a(\cdot)}{k_\delta(t,x-\cdot)} \cdot \nabla_\cdot p(t,\cdot,x) \big\rangle, \;\;\; \delta > 0.
\]
In what follows, we use function $\mathcal N_\delta$ (and its counterpart $\hat{\mathcal N}_\delta$, see Section \ref{apr_GB_sect}) with essentially the same purpose as J.\,Nash did himself in \cite{N}.

\begin{proposition}
\label{TimeIndCorollary2}
If $\delta = c_4$ then there exists a generic constant $c_0$ such that
\[
	\mathcal N_\delta(t,x) \leq \frac{c_0}{t}, \quad (t,x) \in ]0 , \infty[ \times \mathbb{R}^d.
\]
\end{proposition}

\begin{proof}
Write $\mathcal N_\delta = \big\langle \nabla p \cdot \frac{a}{k_\delta} \cdot \nabla p \big\rangle$. Integrating by parts and using the equation $\big( \partial_t + A^\varepsilon \big) p(t,\cdot,x) = 0$, we have
\[
	\mathcal N_\delta = \big\langle -\partial_t p, \frac{p}{k_\delta} \big\rangle + \big\langle \nabla p \cdot \frac{ap}{k_\delta^2} \cdot \nabla k_\delta \big\rangle.
\]
Let us show that the RHS is finite. By \eqref{UGB_p}, \eqref{GB_pt} and by our choice of $\delta$,
\[
	\big| \langle -\partial_t p, \frac{p}{k_\delta} \rangle\big| \leq c_3 c_5 t^{-1} \big \langle \frac{k_{c_6} k_{c_4}}{k_\delta} \big \rangle = \frac{c_3c_5}{t};
\]
Due to (${\rm UGB}^p$) and a \textit{qualitative} bound $|\nabla_xp(t,x,y)|\leq Ct^{-1/2}k_c(t,x,y)$ (i.e.\,the constants $C$, $c$  depend on $\varepsilon$), we have $|\big\langle \nabla p \cdot \frac{ap}{k_\delta^2} \cdot \nabla k_\delta \big\rangle|<\infty$ and hence $\mathcal N_\delta<\infty$.

By quadratic inequalities and \eqref{UGB_p},
\begin{align*}
	\big |\big \langle \nabla p \cdot \frac{ap}{k_\delta^2} \cdot \nabla k_\delta \big \rangle \big |  
		 & \leq c_3 \mathcal N_\delta^\frac{1}{2} \big\langle \nabla k_\delta \cdot \frac{a}{k_\delta} \left( \frac{k_{c_4}}{k_\delta} \right)^2 \cdot \nabla k_\delta \big\rangle^\frac{1}{2}, 
		\end{align*}
\begin{align*}		
\langle \nabla k_\delta \cdot \frac{a k_{c_4}^2}{k_\delta^3} \cdot \nabla k_\delta \rangle & \leq \xi \big \langle \frac{(\nabla k_\delta)^2}{k_\delta} \big \rangle = \frac{\xi d}{2\delta}\frac{1}{t} < \frac{d}{2}\frac{1}{t}.
\end{align*}
and so
\begin{align*}
\mathcal N_\delta  \leq 2 \big \langle -\partial_t p, \frac{p}{k_\delta} \big \rangle + c_3^2 \big \langle \nabla k_\delta \cdot \frac{a}{k_\delta} \cdot \nabla k_\delta \big \rangle \leq \frac{c_0}{t}, \quad \text{ where }c_0 = 2 c_3c_5 + \frac{d}{2}.
\end{align*}
\end{proof}

\section{Proof of Theorem \ref{apr_GB}}

\label{apr_GB_sect}

\subsection{Auxiliary estimates} For a given $\lambda>0$, denote
\[
k_\lambda := k_\lambda (\tau-s,y-\cdot) \quad \text{ and } \quad \hat{k}_\lambda := k_\lambda (t-\tau,x-\cdot),\qquad s<\tau<t
\]
and $$\big \langle \frac{(\nabla k_\lambda )^2}{k_\lambda} \big \rangle := \big \langle \frac{(\nabla_\cdot k_\lambda (\tau-s,y-\cdot) )^2}{k_\lambda (\tau-s,y-\cdot)} \big \rangle.$$
The next three facts are evident:

\begin{enumerate}[label=($\mathbf{a_1}$)]
\item $$
\big \langle \frac{(\nabla k_\lambda )^2}{k_\lambda} \big \rangle  = \frac{d}{2\lambda} \frac{1}{\tau-s} =\big \langle \big(\frac{y-\cdot}{2\lambda (\tau-s)}\big )^2 k_\lambda (\tau-s,y-\cdot) \big \rangle,
$$
$$
\big \langle \frac{(\nabla \hat{k}_\lambda )^2}{\hat{k}_\lambda} \big \rangle  = \frac{d}{2\lambda} \frac{1}{t-\tau}.
$$
\end{enumerate}

\begin{enumerate}[label=($\mathbf{a_2}$)]
\item If $\lambda < \lambda_1$, then $k_\lambda \leq \big(\frac{\lambda_1}{\lambda} \big)^{\frac{d}{2}} k_{\lambda_1}.$
\end{enumerate}

\begin{enumerate}[label=($\mathbf{a_3}$)]
\item If $2\delta > c_4$, then $$\frac{k_{c_4}^2}{k_\delta} = \big( \frac{\delta^2}{(2\delta-c_4)c_4} \big)^{\frac{d}{2}} k_{\frac{\delta c_4}{2\delta-c_4}}.$$
\end{enumerate}

\medskip

\begin{enumerate}[label=($\mathbf{a_4^-}$)]
\item
$\left\{ \begin{array}{l} 0 < 2\delta < \lambda \\ 0 < \varepsilon <1 \\ 0 < \tau-s < (t -s) \varepsilon \end{array} \right. 
\Rightarrow \quad
\left\{ \begin{array}{l} 
\hat{k}_\lambda^2 k_\delta \leq c_-^2 k_{\frac{\lambda \delta}{\lambda -2 \delta}} \cdot k_\lambda^2 (t-s,x-y), 
\\ [3mm]
\text{where }c_-:= (1-\varepsilon)^{-d/2} \big(\frac{\lambda}{\lambda-2\delta} \big)^{d/4}. \end{array} \right. $
\end{enumerate}
\begin{enumerate}[label=($\mathbf{a_4^+}$)]
\item 
$\left\{ \begin{array}{l} 0 < 2\delta < \lambda \\ \frac{\lambda}{2 (\lambda -\delta)} < \varepsilon < 1 \\ (t-s) \varepsilon < \tau-s < t-s \end{array}  \right. 
\Rightarrow \quad
\left\{ \begin{array}{l} 
\hat{k}_\lambda k_{2\delta}^2 \leq c_+^2 \hat{k}_\frac{\lambda}{r} \cdot k_\lambda^2 (t-s,x-y) , 
\\ [3mm]
\text{where }c_+:= \varepsilon^{-d/2} \big(\frac{\lambda}{2 \delta} \big)^{d/2} r^{-d/2} , r = \frac{2(\lambda-\delta)\varepsilon-\lambda}{\lambda-2 \delta \varepsilon} . \end{array} \right. $
\end{enumerate}

\begin{proof}[Proof of $(\mathbf{a_4^-})$]
Using $a b \leq a^2 + 4^{-1}b^2$ and $t-\tau \geq (1-\varepsilon)(t-s)$ we have, for any $\alpha \in \mathbb{R}^d$, $\alpha \neq 0$,
\begin{align*}
& e^{\alpha \cdot (x-y)} \hat{k}^2_\lambda k_\delta = e^{\alpha \cdot (x-\cdot)} \hat{k}^2_\lambda e^{\alpha \cdot (\cdot -y)} k_\delta \\
& \leq (1-\varepsilon)^{-d} \big(4 \pi \lambda (t-s)\big)^{-d} e^{\alpha^2 \frac{\lambda}{2} (t-\tau)} \cdot \big(4 \pi \delta (\tau -s)\big)^{-d/2} e^{\alpha^2 \frac{\lambda}{2} (\tau-s)} e^{- \frac{|\cdot - y|^2}{4 (\tau-s)}\big(\frac{1}{\delta} -\frac{2}{\lambda} \big)} \\
& = (1-\varepsilon)^{-d} \big(\lambda/(\lambda-2 \delta) \big)^{d/2} k_{\frac{\lambda \delta}{\lambda - 2 \delta}} \cdot \big( 4\pi \lambda (t-s)\big)^{-d} e^{\alpha^2 \frac{\lambda}{2} (t-s)} ;
\end{align*}
Therefore,
\[
\hat{k}^2_\lambda k_\delta \leq (1-\varepsilon)^{-d} \big(\lambda/(\lambda-2 \delta) \big)^{d/2} k_{\frac{\lambda \delta}{\lambda - 2 \delta}} \cdot \big( 4\pi \lambda (t-s)\big)^{-d} e^{-\alpha \cdot (x-y) + \alpha^2 \frac{\lambda}{2}(t-s)}
\]
Set $\alpha = \frac{x-y}{\lambda (t-s)}$.
\end{proof}

\begin{proof}[Proof of $(\mathbf{a_4^+})$]
Using $a b \leq a^2 + 4^{-1}b^2$ and $\varepsilon (t-s) \leq \tau -s$ we have, for any $\alpha \in \mathbb{R}^d$, $\alpha \neq 0$ and $r \in ]0 , 1 [ ,$
\begin{align*}
& e^{\alpha \cdot (x-y)} \hat{k}_\lambda k_{2 \delta}^2 = e^{\alpha \cdot (\cdot - y)}k_{2 \delta}^2 e^{\alpha \cdot (x-\cdot)} \hat{k}_\lambda \\
& \leq \varepsilon^{-d}(\lambda /(2 \delta))^d \big(4 \pi \lambda (t-s)\big)^{-d} e^{\alpha^2 \delta (\tau -s)} \cdot  \big(4\pi \lambda (t-\tau)\big)^{-d/2} e^{\alpha \cdot (x-\cdot) - \frac{|x-\cdot|^2}{4 \lambda(t-\tau)}(1-r+r)} \\
& \leq \varepsilon^{-d} (\lambda/(2 \delta)^d  r^{-d/2} \hat{k}_\frac{\lambda}{r} \cdot \big(4 \pi \lambda (t-s)\big)^{-d} e^{\alpha^2 \delta (\tau-s) +\alpha^2 \frac{\lambda}{1-r} (t-\tau)} ;
\end{align*}
Using $t-\tau \leq (1-\varepsilon)(t-s)$ and taking into account our choice of $r$ and $\varepsilon ,$ we have
\begin{align*}
\delta (\tau-s) & + \frac{\lambda}{1-r} (t-\tau) = \delta (t-s) + \big(\frac{\lambda}{1-r} -\delta \big)(t-\tau) \\
& \leq \delta (t-s) + \big(\frac{\lambda}{1-r} -\delta \big) (1-\varepsilon)(t-s) = \frac{\lambda}{2} (t-s).
\end{align*}
Therefore
\[
\hat{k}_\lambda k_{2 \delta}^2 \leq \varepsilon^{-d} (\lambda/(2 \delta)^d  r^{-d/2} \hat{k}_\frac{\lambda}{r} \cdot \big(4 \pi \lambda (t-s)\big)^{-d} e^{-\alpha \cdot (x-y) + \alpha^2 \frac{\lambda}{2} (t-s)} .
\]
Set $\alpha = \frac{x-y}{\lambda (t-s)}$. 
\end{proof}

\subsection{Nash's function $\hat{\mathcal N}_\delta$} Let $p(t,x,y)$ denote the heat kernel of $\partial_t + A^\varepsilon$, $A^\varepsilon \equiv -\nabla \cdot a_\varepsilon \cdot \nabla$. 
Put for brevity $a \equiv a_\varepsilon$. Define 
$$
\hat{\mathcal{N}}_\delta(t-\tau,\tau-s,x,y):= \bigg \langle \nabla_\cdot p(\tau-s,\cdot,y) \cdot \frac{a(\cdot) k_\lambda(t-\tau,x,\cdot)}{k_{2\delta}^2 (\tau-s,y,\cdot)} \cdot \nabla_\cdot p(\tau-s,\cdot,y) \bigg \rangle,
$$
for all $s<\tau<t$, $x,y \in \mathbb R^d$.

\begin{proposition}
\label{nash_prop2}
Let $c_4, c_6 <2\delta<\lambda$, fix $0<\varepsilon<1$. There exists a generic constant $\hat{c}_0$ such that
\begin{equation*}
\hat{\mathcal N}_\delta(t-\tau,\tau-s,x,y) \leq \frac{\hat{c}_0}{t-\tau}
\end{equation*}
for all $t>s$, $(t-s)\varepsilon<\tau-s<t-s$, $x,y \in \mathbb R^d$.
\end{proposition}

\begin{proof}
Write $\hat{\mathcal{N}}_\delta = \big\langle \nabla p \cdot \frac{a \hat{k}_\lambda}{k_{2\delta}^2} \cdot \nabla p \big\rangle$. Integrating by parts and using the equation $\big( \partial_\tau + A^\varepsilon \big) p(\tau-s,\cdot,y) = 0$, we obtain
$$
\hat{\mathcal N}_\delta = \big\langle -\partial_\tau p , \frac{\hat{k}_\lambda p}{k_{2\delta}^2}\big\rangle - \big\langle \nabla p \cdot \frac{a p}{k_{2\delta}^2} \cdot \nabla \hat{k}_\lambda\big\rangle + 2 \big\langle \nabla p \cdot \frac{a  p \hat{k}_\lambda}{k_{2\delta}^3} \cdot \nabla k_{2\delta}\big\rangle.
$$
By quadratic inequalities,
\begin{align*}
|\big \langle \nabla p \cdot \frac{a p}{k_{2\delta}^2} \cdot \nabla \hat{k}_\lambda \big \rangle| & \leq \frac{1}{4} \hat{\mathcal{N}}_\delta +\big \langle \nabla \hat{k}_\lambda \cdot \frac{a p^2}{k_{2\delta}^2 \hat{k}_\lambda} \cdot \nabla \hat{k}_\lambda \big \rangle
\\
& \equiv \frac{1}{4} \hat{\mathcal{N}}_\delta +M_1,\\
2|\big \langle \nabla p \cdot \frac{a p \hat{k}_\lambda}{k_{2\delta}^3} \cdot \nabla k_{2\delta}  \big \rangle| &\leq \frac{1}{4} \hat{\mathcal{N}}_\delta + 4\big \langle \nabla k_{2\delta} \cdot \frac{a p^2 \hat{k}_\lambda }{k_{2\delta}^4} \cdot \nabla k_{2\delta}  \big \rangle \\
& \equiv \frac{1}{4} \hat{\mathcal{N}}_\delta + 4 M_2.
\end{align*}
Therefore,
\begin{equation}
\label{BulEquation}
\tag{$\ast$}
\hat{\mathcal{N}}_\delta \leq 2\big\langle -\partial_\tau p , \frac{\hat{k}_\lambda p}{k_{2\delta}^2}\big\rangle + 2M_1 + 8M_2.
\end{equation}
Let us estimate the terms in the RHS of \eqref{BulEquation}.

 By \eqref{UGB_p}, \eqref{GB_pt} and by our choice of $\delta$,
\begin{align*}
	\big| \big\langle -\partial_\tau p , \frac{\hat{k}_\lambda p}{k_{2\delta}^2}\big\rangle \big| & \leq c_3 c_5 (\tau-s)^{-1} \big \langle \frac{k_{c_6} k_{c_4} \hat{k}_\lambda}{k_{2\delta}^2} \big \rangle \\
	& \leq  c_3c_5(\tau-s)^{-1}\biggl(\frac{(2\delta)^2}{c_4c_6} \biggr)^{\frac{d}{2}}\langle \hat{k}_\lambda \rangle = c_3c_5(\tau-s)^{-1}\biggl(\frac{(2\delta)^2}{c_4c_6} \biggr)^{\frac{d}{2}}. 
\end{align*}
Taking into account that $\tau-s > \varepsilon (t-s) \Rightarrow \frac{1}{\tau-s} < \frac{1-\varepsilon}{\varepsilon}  \frac{1}{t-\tau}$, we thus obtain
$$
	\big| \big\langle -\partial_\tau p , \frac{\hat{k}_\lambda p}{k_{2\delta}^2}\big\rangle \big| \leq c_3c_5\biggl(\frac{(2\delta)^2}{c_4c_6} \biggr)^{\frac{d}{2}}\frac{1-\varepsilon}{\varepsilon}  \frac{1}{t-\tau}.
$$

Next, using  ($\mathbf{a_1}$)-($\mathbf{a_3}$), we have:
\begin{align*}
M_1 & \leq \xi c_3^2 \left\langle \left(\frac{k_{c_4}}{k_{2\delta}} \right)^2 \frac{(\nabla \hat{k}_\lambda)^2}{\hat{k}_\lambda}       \right\rangle \\
& \leq \xi c_3^2 \left(\frac{2\delta}{c_4} \right)^d \big \langle \frac{(\nabla \hat{k}_\lambda)^2}{\hat{k}_\lambda}  \big \rangle \\
& = \xi c_3^2 \left(\frac{2\delta}{c_4} \right)^d \frac{d}{2\lambda} \frac{1}{t-\tau}.
\end{align*}
\[
M_2 \leq \xi c_3^2 \bigg\langle \left(\frac{k_{c_4}}{k_{2\delta}} \right)^2 \hat{k}_\lambda (\nabla \log k_{2\delta} )^2 \bigg\rangle,
\]
where
\begin{align*}
\left(\frac{k_{c_4}}{k_{2\delta}} \right)^2 & = \left(\frac{2\delta}{c_4} \right)^d \exp \bigg[ - \frac{|y-\cdot |^2}{4 (\tau-s)} \bigg(\frac{1}{c_4} - \frac{1}{2\delta} \bigg)2 \bigg] \\
& = \left(\frac{2\delta}{c_4} \right)^d \exp \bigg[ - \frac{|y-\cdot |^2}{4\gamma (\tau-s)} \bigg], \qquad \gamma:= \frac{\delta c_4}{2\delta -c_4},
\end{align*}
\begin{align*}
(\nabla \log k_{2\delta})^2 &= \bigg(\frac{y-\cdot}{2(2\delta)(\tau-s)} \bigg)^2 =\frac{|y-\cdot|^2}{4\gamma (\tau-s)} \frac{\gamma}{(2\delta)^2}  \frac{1}{\tau-s}.
\end{align*}
Since $0<\eta < e^\eta$, we have therefore
\[
\bigg\langle \left(\frac{k_{c_4}}{k_{2\delta}} \right)^2 \hat{k}_\lambda (\nabla \log k_{2\delta} )^2 \bigg\rangle \leq \left(\frac{2\delta}{c_4} \right)^d \frac{\gamma}{(2\delta)^2} \frac{1}{\tau-s} \langle \hat{k}_\lambda \rangle,
\]
and so
$$
M_2 \leq \xi c_3^2 \left(\frac{2\delta}{c_4} \right)^d \frac{c_4}{(2\delta-c_4)4\delta} \frac{1-\varepsilon}{\varepsilon} \frac{1}{t-\tau}.
$$
Substituting the previous estimates  into (\ref{BulEquation}), we obtain
\[
\hat{\mathcal{N}}_\delta \leq 2\,c_3c_5\biggl(\frac{(2\delta)^2}{c_4c_6} \biggr)^{\frac{d}{2}}\frac{1-\varepsilon}{\varepsilon}  \frac{1}{t-\tau} +c_3^2 \left(\frac{2\delta}{c_4} \right)^d \left(2 \cdot \frac{\xi d}{2\lambda} + 8\cdot\frac{2\xi}{4\delta} \cdot \frac{c_4}{2\delta-c_4} \cdot \frac{1-\varepsilon}{\varepsilon} \right) \frac{1}{t-\tau},
\]
as claimed.
\end{proof}

\subsection{Proof of the upper bound} For brevity, $b \equiv b_\varepsilon$. We iterate the Duhamel formula
\[
u(t-s,x,y) = p(t-s,x,y) - \int_s^t \langle u(t-\tau,x, \cdot) b(\cdot) \cdot \nabla_\cdot p(\tau-s,\cdot,y) \rangle d\tau.
\]
We obtain the series
\[
l(t-s,x,y):=\sum_{n=0}^\infty (-1)^n u_n (t-s,x,y),
\]
where $u_0 (t-s,x,y):= p(t-s,x,y)$ and, for $n=1,2,\dots,$
\[
 u_n(t-s,x,y):=\int_s^t \langle u_{n-1} (t-\tau,x,\cdot) b(\cdot) \cdot \nabla_\cdot p(\tau-s,\cdot,y) \rangle d\tau.
\]
In particular,
\[
u_1 (t-s,x,y)= \int_s^t \langle p(t-\tau,x,\cdot) b(\cdot) \cdot \nabla_\cdot p(\tau-s,\cdot,y) \rangle d\tau,
\]
and so
\[ 
|u_1(t-s,x,y)| \leq c_3 \int_s^t \langle k_{c_4} (t-\tau,x-\cdot) | b(\cdot) \cdot \nabla_\cdot p(\tau-s,\cdot,y) | \rangle d\tau.
\]
\textit{Suppose that we are able to find generic* constants $h >0$ and $C_h <1$ such that the bound:
\begin{equation}
\int_s^t \langle k_{c_4} (t-\tau,x-\cdot) | b(\cdot) \cdot \nabla_\cdot p(\tau-s,\cdot,y) | \rangle d\tau \leq C_h k_{c_4} (t-s,x-y)
\label{star_b_star_N}
\tag{$\star^b \star^N$}
\end{equation}
is valid for all $x,y \in \mathbb{R}^d$ and $0<t-s \leq h$.}

Then $|u_1 (t-s,x,y)| \leq c_3 C_h k_{c_4} (t-s,x-y)$, and by induction,
\[
|u_n (t-s,x,y)| \leq c_3 \big(C_h\big)^n k_{c_4} (t-s,x-y).
\]
Therefore, for all $0 < t-s \leq h$ and all $x,y \in \mathbb{R}^d$, the series $l(t-s,x,y)$ is well defined and
\[
|l(t-s,x,y) | \leq \frac{c_3}{1-C_h} k_{c_4} (t-s,x-y).
\]
Repeating the standard argument we conclude that $l$ satisfies the Duhamel formula
provided that $0<t-s \leq h$.
Then the uniqueness of $u(t-s,x,y)$ implies
\[
u=l \;\;\; (0<t-s \leq h),
\] 
and the reproduction property of $u$ implies
\[
u(t-s,x,y) \leq \frac{c_3}{1-C_h} e^{(t-s) \omega_h} k_{c_4} (t-s,x-y)
\]
for all $t-s > h$, where $\omega_h = \frac{1}{h} \log \frac{c_3}{1-C_h}$.
Thus, we obtain the upper bound in \eqref{GB_u} of Theorem \ref{apr_GB}.

It remains to prove \eqref{star_b_star_N}. Without loss of generality, $s=0$. Set $b_a^2:=b\cdot a^{-1} \cdot b$ and denote
\[
\langle k_\mu b_a^2 \rangle:=\langle k_\mu(\tau,y-\cdot) b_a^2(\cdot) \rangle, \qquad \langle \hat{k}_\mu b_a^2 \rangle:=\langle k_\mu(t-\tau,x-\cdot) b_a^2(\cdot) \rangle.
\]
Set
\[
I:= \int_0^t \langle k_\lambda (t-\tau,x-\cdot) |b(\cdot) \cdot \nabla_\cdot p(\tau,\cdot,y) | \rangle d\tau.
\]

\begin{lemma}
\label{ELemma} 
Fix $\lambda > \xi$ and select constants $\delta$, $c_4$ such that
\[ 
\lambda > 2\delta > c_4 > \xi.
\]
Let $\frac{\lambda}{2 (\lambda -\delta)} < \varepsilon < 1$, $r = \frac{2 (\lambda - \delta) \varepsilon -\lambda}{\lambda - 2 \delta \varepsilon}$, and let $c_\pm$  be the constants defined in ($\mathbf{a_4^{\pm}}$).
Then, for all $x,y \in \mathbb{R}^d$ and $t>0$,
\[
I \leq (c_- M^- + c_+ M^+ ) k_\lambda (t,x,y),
\]
where 
\begin{align*}
M^-:= & \int_0^{t\varepsilon} \sqrt{\big\langle k_{\frac{\lambda \delta}{\lambda-2\delta}} b_a^2 \big\rangle } \sqrt{\frac{c_0}{\tau}}\; d\tau, \\
M^+:= & \int_{t \varepsilon}^t \sqrt{\big\langle \hat{k}_\frac{\lambda}{r} b_a^2 \big\rangle } \sqrt{\frac{\hat{c}_0}{t-\tau}}\; d\tau.
\end{align*}
\end{lemma}

\begin{proof}
Using quadratic inequality, we bound $\langle \hat{k}_\lambda |b \cdot \nabla p| \rangle^2$ in two ways:
\[
\langle \hat{k}_\lambda |b \cdot \nabla p| \rangle^2 \leq \langle \hat{k}_\lambda^2 k_\delta b_a^2 \rangle \langle \nabla p \cdot \frac{a}{k_\delta} \cdot \nabla p \rangle
\]
and
\[
\langle \hat{k}_\lambda |b \cdot \nabla p| \rangle^2 \leq \langle \hat{k}_\lambda k_{2\delta}^2 b_a^2 \rangle \langle \nabla p \cdot \frac{a \hat{k}_\lambda}{k_{2\delta}^2} \cdot \nabla p \rangle,
\]
and hence
\[
I\equiv \int_0^t \langle \hat{k}_\lambda |b \cdot \nabla p| \rangle \; d \tau \leq  I_\varepsilon^- + I_\varepsilon^+ ,
\]
where
\begin{align*}
I_\varepsilon^- := & \int_0^{t \varepsilon} \sqrt{\langle \hat{k}_\lambda^2 k_\delta b_a^2 \rangle} \sqrt{\langle \nabla p \cdot \frac{a}{k_\delta} \cdot \nabla p \rangle} \; d \tau \\
I_\varepsilon^+ := & \int_{t\varepsilon}^t \sqrt{\langle \hat{k}_\lambda k_{2\delta}^2 b_a^2 \rangle}  \sqrt{\langle \nabla p \cdot \frac{a \hat{k}_\lambda}{k_{2\delta}^2} \cdot \nabla p \rangle} \; d \tau
\end{align*}
Now the assertion of Lemma \ref{ELemma}  follows directly from ($\mathbf{a_4^\mp}$) and Propositions \ref{TimeIndCorollary2} and \ref{nash_prop2}. (Here we apply Propositions \ref{TimeIndCorollary2} with $\delta$ chosen as in Proposition \ref{nash_prop2}, but it is not difficult to see, using ($\mathbf{a_3}$), that its proof works for all $\delta>\frac{c_4}{2}$ although with different generic constant $c_0$.)
\end{proof}

It remains to note that both $M_+$, $M_-$ in Lemma \ref{ELemma} are majorated by $c\,n_e(b,h)$ for appropriate multiple $c>0$. Provided that $n_e(b,h)$ is sufficiently small, i.e.\,so that $C_h:=(c_-  + c_+ )c n_e(b,h)<1$, we obtain \eqref{star_b_star_N}.

\subsection{Proof of the lower bound} 
The analysis of the previous section and the Gaussian upper bound \eqref{UGB_p} of Theorem \ref{thm_p} yield for $|x-y|^2 \leq t \leq h$
\begin{align}
u(t,x,y) &  \geq p(t,x,y)- \sum_{n \geq 1} |u_n (t,x,y)| \notag \\
& \geq c_1 k_{c_2} (t,x-y) - \frac{c_3 C_h}{1-C_h} k_{c_4} (t,x-y) \notag \\
& \geq \left(c_1 c_2^{-\frac{d}{2}} e^{-\frac{1}{4c_2}} - \frac{c_3 C_h}{1-C_h} c_4^{-\frac{d}{2}} \right) ( 4\pi t )^{-\frac{d}{2}} \notag \\
& \equiv r t^{-\frac{d}{2}}, \label{lower_diagonal} \tag{$\ast\ast$}
\end{align}
where $r>0$ provided that $C_h$ is small enough, i.e.
$
\frac{C_h}{1-C_h} < \frac{c_1}{c_3} \left( \frac{c_4}{c_2} \right)^{\frac{d}{2}} e^{-\frac{1}{4c_2}}.
$

Now the standard argument (``small gains yield large gain'', see e.g.\,\cite[Theorem 3.3.4]{Da}) yields for all $x,y\in\mathbb{R}^d$, $t>0$, 
\begin{equation*}
u(t,x,y) \geq  re^{t\nu_h}t^{-\frac{d}{2}} \exp \left(-\frac{|x-y|^2}{4c_2 t} \right),\quad \nu_h=\frac{1}{h}\log r.
\end{equation*}

The proof of Theorem \ref{apr_GB} is completed.

\bigskip

\section{Proof of Proposition \ref{prop1}}

\label{prop1_sect}

\textbf{1.~}Let $\mathbf{1}_\varepsilon$, $\varepsilon>0$ be the indicator of $\{x \in \mathbb R^d \mid |x| \leq \varepsilon^{-1}, |b(x)| \leq \varepsilon^{-1}\}$. Define 
$$
b_\varepsilon:=E_{\nu_\varepsilon}(\mathbf{1}_\varepsilon b),
$$
where, recall, $E_\nu\equiv e^{\nu\Delta}$, and $\varepsilon$, $\nu_\varepsilon>0$.

Define also
$
(b^2)_\varepsilon=E_{\nu_\varepsilon}(\mathbf{1}_\varepsilon b^2)
$
and set $g_{1,\varepsilon}:=b_\varepsilon -\mathbf{1}_\varepsilon b$ and $g_{2,\varepsilon}:=|(b^2)_\varepsilon-\mathbf{1}_\varepsilon b^2|$. 

In what follows, we select $\{\nu_\varepsilon \}$ so that $\nu_\varepsilon \downarrow 0$ sufficiently rapidly as $\varepsilon\downarrow 0$ so that $\|g_{1,\varepsilon}\|_2\leq \varepsilon$ and $\|g_{2,\varepsilon}\|_q\leq \varepsilon^2$ for some $q\geq d$.
Note that $(b^2)_\varepsilon\leq g_{2,\varepsilon}+b^2$. Since $\|\mathbf{1}_{B(0,R)}(b_\varepsilon - b)\|_2\leq \|g_{1,\varepsilon}\|_2+\|\mathbf{1}_{B(0,R)}(\mathbf{1}_\varepsilon b-b)\|_2$, we have
$$
b_\varepsilon\rightarrow b \quad \text{ strongly in } [L^2_\loc]^d.
$$
The Nash norm of $b_\varepsilon$ is controlled by the Nash norm of $b$:

\begin{lemma} \label{lem1}
$n_e(b_\varepsilon,h) \leq n_e(b,h) + c_d h^{\frac{1}{4}}\varepsilon$, $\varepsilon>0$.
\end{lemma}
\begin{proof}
Clearly, $(b_\varepsilon)^2\leq (b^2)_\varepsilon$, and so
\begin{align*} 
n_e(b_\varepsilon,h) & \equiv \sup_{x \in \mathbb R^d}\int_0^h \sqrt{ e^{t\Delta}(b_\varepsilon)^2(x)} \frac{dt}{\sqrt{t}} \\
& \leq n_e(b,h) + \sup_{x \in \mathbb R^d}\int_0^h\sqrt{ e^{t\Delta}g_{2,\varepsilon}(x)} \frac{dt}{\sqrt{t}},
\end{align*}
where  
\begin{align*}
\sup_{x \in \mathbb R^d}\int_0^h\sqrt{ e^{t\Delta}g_{2,\varepsilon}(x)} \frac{dt}{\sqrt{t}} & \leq \int_0^h\sqrt{\|e^{t\Delta}g_{2,\varepsilon}\|_\infty}\frac{dt}{\sqrt{t}} \leq C_d\int_0^h\sqrt{t^{-\frac{d}{2q}}\|g_{2,\varepsilon}\|_q}\frac{dt}{\sqrt{t}} \\
& \leq \sqrt{\|g_{2,\varepsilon}\|_q} C_d\frac{2}{1-\frac{d}{2q}}h^{\frac{1}{2}-\frac{d}{4q}} \leq 4C_d h^\frac{1}{4}\varepsilon.
\end{align*}
\end{proof}

\textbf{2.\,}Now we can give

\begin{proof}[Proof of Proposition \ref{prop1}] Set $\delta:=c_4$. We will construct $(b \cdot \nabla)_1$ and prove 
\begin{equation}
\label{__bAest}
\|(b \cdot \nabla)_1 g\|_1 \leq \eta\|(\zeta + A_1)g\|_1, \quad g \in D(A_1),
\end{equation}
with $\eta:=\frac{1}{1-e^{-\Real\zeta h}}\sqrt{\frac{c_0}{\sigma\delta}}\;n_e(b,h\delta),$
for all $\Real\zeta > 0$, so taking $\zeta:=\mu>0$ we obtain the assertion of the proposition.
 
\smallskip

\textit{Step 1.} Put $B_1^\varepsilon:=[b_\varepsilon \cdot\nabla\upharpoonright C_c^1]^{\rm clos}_{L^1\to L^1}$ of domain $\mathcal{W}^{1,1}$, and $$T_1^\varepsilon:=B_1^\varepsilon (\zeta+A_1^\varepsilon)^{-1}\in \mathcal{B}(L^1),$$ where, recall, $A^\varepsilon_1:=-\nabla \cdot a_\varepsilon \cdot \nabla$, $a_\varepsilon \equiv E_\varepsilon a$, $D(A_1^\varepsilon)=\mathcal W^{2,1}.$ Since $B_1^\varepsilon$ is closed, 
we can write
\[
T_1^\varepsilon f(x)=\int_0^\infty e^{-\zeta t}B_1^\varepsilon e^{-tA_1^\varepsilon}f(x)dt=\int_0^\infty e^{-\zeta t}\langle b_\varepsilon(x)\cdot\nabla_x p_\varepsilon(t,x,\cdot)f(\cdot)\rangle dt,\quad f\in \mathcal{W}^{1,1}.
\] 
Denote $\mu:=\Real \zeta$. We have
\begin{align*}
\|T_1^\varepsilon f\|_1 &\leq \sum_{j=0}^\infty e^{-j\mu h}\int_{jh}^{(j+1)h}\|B_1^\varepsilon e^{-tA_1^\varepsilon}f\|_1dt\\
&= \sum_{j=0}^\infty e^{-j\mu h}\int_0^{ h}\|B_1^\varepsilon e^{-tA_1^\varepsilon} e^{-jhA_1^\varepsilon}f\|_1dt.
\end{align*}
By the Fubini Theorem and the Cauchy-Bunyakovsky inequality,
\begin{align*}
\int_0^h\|B_1^\varepsilon e^{-tA^\varepsilon} e^{-jhA_1^\varepsilon}f\|_1dt 
&\leq \big\langle\int_0^h \langle |b_\varepsilon(x)\cdot\nabla_x p_\varepsilon(t,x,y)|\rangle_x dt|e^{-jhA_1^\varepsilon}f(y)| \big\rangle_y\\
&\leq  \sup_{y\in\mathbb{R}^d}\int_0^h \langle |b_\varepsilon(x)\cdot\nabla_x p_\varepsilon(t,x,y)|\rangle_x dt\|f\|_1 \\
&\leq \sup_{y\in\mathbb{R}^d}\int_0^h \sqrt{\langle k_\delta(t,x-y)(b_\varepsilon\cdot a^{-1}_\varepsilon\cdot b_\varepsilon) (x)\rangle_x}\sqrt{\mathcal{N}_\delta(t,y)}dt\|f\|_1,
\end{align*}
where $\mathcal{N}_\delta(t,y) \equiv \big\langle \nabla_xp_\varepsilon(t,x,y)\cdot\frac{a_\varepsilon(x)}{k_\delta(t,x-y)}\cdot \nabla_xp_\varepsilon(t,x,y)\big\rangle_x\leq \frac{c_0}{t}$ by Proposition \ref{TimeIndCorollary2}.
Therefore,
\begin{align*}
\int_0^h\|B_1^\varepsilon e^{-tA_1^\varepsilon} e^{-jhA_1^\varepsilon}f\|_1dt &\leq \sqrt{\frac{c_0}{\sigma\delta}}\;n_e(b_\varepsilon,h\delta)\|f\|_1 \\
& (\text{we are applying lemma above}) \\
& \leq \sqrt{\frac{c_0}{\sigma\delta}} \bigl(n_e(b,h\delta) + c_d h^\frac{1}{4}\delta^{\frac{1}{4}}\varepsilon\bigr)\|f\|_1.
\end{align*}
Thus, 
\[
\|T_1^\varepsilon f\|_{1}\leq \eta_\varepsilon \|f\|_1, \quad f \in L^1, \quad \eta_\varepsilon:=\eta + \tilde{c} \varepsilon, \quad \Real \zeta> 0.
\]

\smallskip

\textit{Step 2.}~ Set $Tf:=b \cdot \nabla (\zeta+A)^{-1}f$, $f \in L^2$ and note that $\nabla (\zeta+A^\varepsilon)^{-1} \rightarrow \nabla(\zeta+A)^{-1} $ strongly in $[L^2]^d$. 
[The proof is standard:
For $1 \leq i \leq d$, $f \in W^{-1,2}$, $\|\nabla_i (\zeta+A^\varepsilon)^{-1}f - \nabla_i (\zeta+A)^{-1}f\|_{2} =: M_\varepsilon(f)$, 
\begin{align*}
M_\varepsilon(f) & := \|\nabla_i(\zeta+A^\varepsilon)^{-1}\nabla\cdot (a-a_\varepsilon)\cdot\nabla(\zeta+A)^{-1}f\|_{2} \\
& \leq\|\nabla_i(\zeta+A^\varepsilon)^{-1}\nabla\|_{2\rightarrow 2}\|(a-a_\varepsilon)\cdot\nabla(\zeta+A)^{-1}f\|_{2},
\end{align*}
where $\|\nabla_i(\zeta+A^\varepsilon)^{-1}\nabla\|_{2\rightarrow 2} \leq \|\nabla (\zeta+A^\varepsilon)^{-\frac{1}{2}} \|^2_{2 \rightarrow 2} \leq C$, $C \neq C(\varepsilon)$ and
$\|(a-a_\varepsilon)\cdot\nabla(\zeta+A)^{-1}f\|_{2} \rightarrow 0$ (e.g.\,using the Dominated Convergence Theorem),
so $M_\varepsilon(f) \rightarrow 0$ as $\varepsilon \downarrow 0$, in particular, for $f \in L^2$.]

Therefore, since $b_\varepsilon \rightarrow b$ strongly in $[L^2_{\loc}]^d$, 
\begin{equation}
\label{__Tconv}
T^\varepsilon f \rightarrow Tf \quad \text{strongly in $L^1_\loc$ as $\varepsilon\downarrow 0$.} 
\end{equation}
Passing to a subsequence in $\varepsilon$, if necessary, we have $T^\varepsilon f\rightarrow Tf \;\mathcal{L}^d \text{ a.e.}$ Applying Fatou's Lemma, we have by Step 1, for all $f\in L^1\cap L^2$, 
\begin{equation}
\label{__T_liminf}
\|T f\|_1\leq \liminf_\varepsilon\|T^\varepsilon f\|_1 \leq \eta\|f\|_1.
\end{equation}

Let $T_1$ denote the extension of $T\upharpoonright L^1\cap L^2$ by continuity to $L^1$. 

\smallskip

\textit{Step 3.} Since, by Step 2, $\|b\cdot\nabla(\zeta+A)^{-1}f\|_1\leq \eta \|f\|_1$ for all $f\in L^1 \cap L^2$, $\Real \zeta>0$, the operator $B:=b\cdot\nabla\upharpoonright D(A_1)\cap D(A):L^1\to L^1$, and $$\|b\cdot\nabla h\|_1\leq \eta\|(\zeta+A_1)h\|_1, \quad h\in D(A_1)\cap D(A).$$ Since $D(A_1)\cap D(A)$ ($=(1+A)^{-1}[L^1 \cap L^2]$) is a core of $A_1$, $B$ extends by continuity in the graph norm of $A_1$ to $A_1$-bounded operator $(b\cdot\nabla)_1$. 
The proof of Proposition \ref{prop1} is completed.
\end{proof}

\begin{remark}
\label{rem_nonlocal}
The proof above can be extended to non-local operators of the type $\Lambda=(\mu -\nabla \cdot a \cdot \nabla)^{\frac{\alpha}{2}} + b \cdot \nabla$, $1<\alpha<2$, with $b$ in an appropriate modification of the elliptic Nash class.

That is, assume that $b \in [L^2_{\loc}]^d$ satisfies 
$$
\tilde{n}^\alpha(b,\mu)=\sup_{y \in \mathbb R^d}\int_0^\infty e^{- \mu t} \sqrt{e^{t \Delta}|b|^2(y)}\frac{dt}{t^\frac{3-\alpha}{2}}<\infty, \quad \mu>0.
$$

Put $T_1^\varepsilon:=b_\varepsilon \cdot \nabla (\mu + A_1^\varepsilon)^{-\frac{\alpha}{2}}$.
A key bound $\|T^\varepsilon_1 f\|_1 \leq \tilde{\eta}\|f\|_1$, $f \in L^1$ remains valid with $\tilde{\eta}=\delta^\frac{1-\alpha}{2}\sqrt{\frac{c_0}{\sigma}}\tilde{n}^\alpha(b,\mu\delta^{-1})$.
Namely, 
\begin{align*}
\|T_1^\varepsilon f\|_1
& \leq \biggl(\sup_y \int_0^\infty e^{-\mu t}t^{\frac{\alpha}{2}-1} \sqrt{\langle k_\delta(t,y-\cdot)b_a^2(\cdot)\rangle}\sqrt{\mathcal N_\delta(t,y)} dt\biggr) \|f\|_1 \qquad (b_a^2=b \cdot a^{-1}\cdot b)\\
& \leq\delta^\frac{1-\alpha}{2} \sqrt{\frac{c_0}{\sigma}}\tilde{n}^\alpha(b,\mu\delta^{-1})\|f\|_1.
\end{align*}
Above one can replace $\tilde{n}^\alpha(b,\mu)$ by $n^\alpha(b,h):=\sup_{y \in \mathbb R^d}\int_0^h \sqrt{e^{t \Delta}|b|^2(y)}\frac{dt}{t^\frac{3-\alpha}{2}}$.
\end{remark}

\section{Proof of Theorem \ref{nash_apost_thm}}

\label{nash_apost_thm_sect}

In the proof of Proposition \ref{prop1} we established: $T_1^\varepsilon:=b_\varepsilon \cdot \nabla(\zeta+A_1^\varepsilon)^{-1}$, $T_1:=(b \cdot \nabla)_1(\zeta +A_1)^{-1}$, $\Real \zeta>0$ satisfy $T_1 \in \mathcal{B}(L^1)$ and $$\|T^\varepsilon_1\|_{1 \rightarrow 1} \leq \eta+\tilde{c}\varepsilon, \quad \|T_1\|_{1 \rightarrow 1} \leq \eta.$$ 

\begin{proposition}
\label{prop2}
$
T_1=s\mbox{-}L^1\mbox{-}\lim_{\varepsilon \downarrow 0} T_1^\varepsilon.
$
\end{proposition}

\begin{proof}[Proof of Proposition \ref{prop2}]
Under the additional assumption $b^2\in L^1+L^\infty$, the assertion of the proposition is evident (use \eqref{__Tconv} in the proof of Proposition \ref{prop1}). In general one has to employ the separation property of $e^{-tA}$, as is done below.
 
Since $\sup_{\varepsilon>0}\|T_1^\varepsilon\|_{1 \rightarrow 1}, \|T_1\|_{1 \rightarrow 1}<\infty$, it suffices to prove the claimed convergence on $C_c^\infty$. Fix $f\in C_c^\infty$ and then $r>0$ by $B(0,r) \supset \sprt f$.
Since by \eqref{__Tconv} $T_1^\varepsilon f\rightarrow T_1f$ strongly in $L^1_{\loc}$, the required convergence in (\textit{ii}) would follow from \eqref{__T_liminf} once we show that, for every $\theta>0$, there exists $R=R(r,\theta)>0$ such that
$$
\|\mathbf{1}_{B^c(0,R)}T_1^\varepsilon f\|_1 \leq \theta \|f\|_1 \quad \text{ for all $\varepsilon>0$ sufficiently small.}
$$
Here $B^c(0,R):=\mathbb R^d - B(0,R)$.

To prove the latter, we write
\begin{equation*}
\mathbf{1}_{B^c(0,R)}T_1^\varepsilon f(x)=\int_0^\infty e^{-\zeta t}\langle \mathbf{1}_{B^c(0,R)}(x)b_\varepsilon(x)\cdot\nabla_x p_\varepsilon(t,x,\cdot)f(\cdot)\rangle dt,
\end{equation*}
where $p_\varepsilon(t,x,y)=e^{-tA_1^\varepsilon}(x,y)$. Put $\mu:=\Real \zeta$.
Then
\begin{align*}
\|\mathbf{1}_{B^c(0,R)}T_1^\varepsilon f\|_1 & \leq \sum_{j=0}^\infty e^{-j\mu h}\int_{jh}^{(j+1)h}\|\mathbf{1}_{B^c(0,R)} B_1^\varepsilon e^{-tA_1^\varepsilon}f\|_1dt\\
& = \sum_{j=0}^\infty e^{-j\mu h}\int_0^{ h}\|\mathbf{1}_{B^c(0,R)} B_1^\varepsilon e^{-tA_1^\varepsilon} e^{-jhA_1^\varepsilon}f\|_1dt  \\
&= \sum_{j=0}^\infty e^{-j\mu h}\biggl[\int_0^{ h}\|\mathbf{1}_{B^c(0,R)} B_1^\varepsilon e^{-tA_1^\varepsilon} \mathbf{1}_{B(0,mr)} e^{-jhA_1^\varepsilon}f\|_1dt \\
& + \int_0^{ h}\|\mathbf{1}_{B^c(0,R)} B_1^\varepsilon e^{-tA_1^\varepsilon} \mathbf{1}_{B^c(0,mr)} e^{-jhA_1^\varepsilon}f\|_1dt \biggr]=: \sum_{j=0}^\infty e^{-j\mu h}\big[I_j + J_j\big],
\end{align*}
where constant $m \geq 1$ is to be chosen. Arguing as in the proof of Step 1 of the proof of Proposition \ref{prop1} and putting $\delta:=c_4$, we obtain, for all $j \geq 0$, 
\begin{align*}
I_j & \leq \sqrt{\frac{c_0}{\sigma\delta}}\sup_{y \in B(0,mr)} \int_0^h \sqrt{\langle k_\delta(t,y,\cdot)\mathbf{1}_{B^c(0,R)}(\cdot)|b_\varepsilon(\cdot)|^2\rangle} \frac{dt}{\sqrt{t}}\,\|e^{-khA_1^\varepsilon} f\|_1 \\
& \leq \biggl(\sqrt{\frac{c_0}{\sigma\delta}}M_R + 4C_d(h\delta)^{\frac{1}{4}}\varepsilon\biggr)\|f\|_1, 
\end{align*}
where $M_R:=\sup_{y \in B(0,mr)} \int_0^h \sqrt{\langle k_\delta(t,y,\cdot)\mathbf{1}_{B^c(0,R)}(\cdot)|b(\cdot)|^2\rangle} \frac{dt}{\sqrt{t}}$, $R>mr$. 

Clearly, $J_0=0$. For all $j \geq 1$ and $\eta_0=\sqrt{\frac{c_0}{\sigma\delta}}n_e(b,h\delta)$, 
\begin{align*}
J_j & \leq \eta_0 \|\mathbf{1}_{B^c(0,mr)} e^{-jhA_1^\varepsilon}f\|_1 \\
& \text{(we are applying \eqref{UGB_p} to $e^{-jhA_1^\varepsilon}(x,y)$)} \\
& \leq \eta_0 c_3 (4\pi c_4jh)^{-\frac{d}{2}}e^{-\frac{(m-1)^2r^2}{4c_4jh}}\|f\|_1.
\end{align*}

Thus, we have
$$
\|\mathbf{1}_{B^c(0,R)}T_1^\varepsilon f\|_1 \leq \theta\|f\|_1,$$
where
$$
\theta:=\biggl(\sqrt{\frac{c_0}{\sigma\delta}}M_R + 4C_d(h\delta)^{\frac{1}{4}}\varepsilon\biggr)\frac{1}{1-e^{-\mu h}} + C_g \sum_{j=1}^\infty e^{-\mu j h}(jh)^{-\frac{d}{2}}e^{-\frac{(m-1)^2r^2}{4c_4jh}}.
$$
It is clear that selecting $m$ sufficiently large, we can make the second term in the RHS as small as needed. 

We are left to prove the convergence $M_R \rightarrow 0$ as $R \rightarrow \infty$.

$(a_1)$~Fix $n>0$ by $k_\delta(t,z,y) \leq C_n k_\delta(t,z,0)$ for all $t>0$, $z \in B^c(0,(m+n)r)$, $y \in B(0,mr)$. Then
\[
M_R\leq C_n \int_0^h \sqrt{\langle k_\delta(t,0,\cdot)\mathbf{1}_{B^c(0,R)}(\cdot)|b(\cdot)|^2\rangle} \frac{dt}{\sqrt{t}} \quad \forall R>(m+n)r.
\]

$(a_2)$~Due to $b \in \mathbf{N}_e$ the function
$$w_R(t):=\sqrt{\langle k_\delta(t,\cdot,0)\mathbf{1}_{B^c(0,R)}(\cdot)|b(\cdot)|^2\rangle} \frac{1}{\sqrt{t}}
$$
is in $L^1([0,h])$ for every $R \geq 1$. Moreover, it is seen from the definition of $w_R$ that for every $0<t_1<t_2\leq h$,
$
w_R(t_1) \leq C_{t_1,t_2-t_1}w_R(t_2)$, $C_{t_1,t_2-t_1}<\infty.
$
Thus, $w_R(t)$ is finite for all $0<t\leq h$. 

$(a_3)$~$w_R(t) \rightarrow 0$ as $R \rightarrow \infty$ for every $0<t\leq h$. 

Indeed, fix $t\in ]0,h]$. Set $v_R(x):=k_\delta(t,x,0)\mathbf{1}_{B^c(0,R)}(x)|b(x)|^2$. For a.e.\,$x \in \mathbb R^d$, $v_R(x) \downarrow 0$ as $R \uparrow \infty$, and $v_R \leq v_1$ a.e.\,on $\mathbb R^d$ for all $R \geq 1$,
 where $v_1$ is summable. 
Hence by the Dominated Convergence Theorem, $\langle v_R \rangle \rightarrow 0$ as $R \rightarrow \infty$, and so $w_R(t) \rightarrow 0$ as $R \rightarrow \infty$.

$(a_4)$~Due to $(a_3)$ and  $w_R \leq w_1$ for $R \geq 1$,  the Dominated Convergence Theorem yields
$$
\int_0^h w_R(t)dt \rightarrow 0 \quad \text{ as } R \rightarrow \infty.
$$
Thus, $M_R \rightarrow 0$ as $R \rightarrow \infty.$ The proof of Proposition \ref{prop2} is completed.
\end{proof}

\smallskip

We are in position to complete the proof of Theorem \ref{nash_apost_thm}. Recall $\delta:=c_4$.

\smallskip

(\textit{i}) By our assumption on $n_e(b,h\delta)$,  there exists $\lambda_0>0$ such that $$\eta:=\frac{1}{1-e^{-\lambda_0 h}}\sqrt{\frac{c_0}{\sigma\delta}}n_e(b,h\delta)<1.$$  By Proposition \ref{prop1}, $\Lambda_1$ is a closed densely defined operator. Using \eqref{__bAest}, we obtain that 
$$(\zeta+\Lambda_1)^{-1}=(\zeta+A_1)^{-1}(1+T_1)^{-1} \in \mathcal B(L^1), \quad \Real \zeta>\lambda_0.$$ Using \eqref{__A_1_bd}, we obtain
\begin{equation}
\label{__theta_bd}
\|(\zeta+\Lambda_1)^{-1}\|_{1\to 1}\leq \frac{M}{|\zeta|(1-\eta)}, \quad \Real \zeta > \lambda_0,
\end{equation}
completing the proof of the first part of assertion (\textit{i}).

To prove the second part of (\textit{i}), note that, in view of \eqref{__theta_bd}, the resolvent $\zeta \mapsto (\zeta+\lambda_0+\Lambda_1)^{-1}=\Theta(\zeta+\lambda_0)$ is holomorphic in the right-half plane $\Real\zeta>0$ and in $|\zeta-\zeta_0|<\sqrt{2}\bigl(\frac{M}{1-\eta}-1\big)|\zeta_0|$ for every $\zeta_0$ with $\Real \zeta_0=0$ (see, if needed, the argument in \cite[Ch.\,IX, sect.\,10]{Y}). Thus,  $e^{-z(\lambda_0+\Lambda_1)}$ is holomorphic in the sector
$$\{z \in \mathbb C \mid |{\rm arg\, z}|<\frac{\pi}{2} - \theta_{\lambda_0}\}, \quad \text{ where } \tan \theta_{\lambda_0}=\sqrt{2}\biggl(\frac{M}{1-\eta}-1\bigg).$$ 
This completes the proof of assertion (\textit{i}).

\smallskip

(\textit{ii}) The claimed approximation $\{b_\varepsilon\}$ was constructed in the proof of Proposition \ref{prop1}. Let us show that
$$
(\lambda+\Lambda_1^\varepsilon)^{-1} \rightarrow (\lambda+\Lambda_1)^{-1} \quad \text{strongly in $L^1$ as $\varepsilon \downarrow 0$},
$$
which, by a standard result, implies the convergence of the semigroups.

 Since $(\lambda+\Lambda_1^\varepsilon)^{-1}=(\lambda+A_1^\varepsilon)^{-1}(1+T^\varepsilon_1)^{-1}$,  $(\lambda+\Lambda_1)^{-1}=(\lambda+A_1)^{-1}(1+T_1)^{-1}$, it suffices to show that 1) $T_1^\varepsilon \rightarrow T_1 $ and 2) $(\lambda+A_1^\varepsilon)^{-1} \rightarrow (\lambda+A_1)^{-1}$ strongly in $L^1$ as $\varepsilon \downarrow 0$. 1) is Proposition \ref{prop2}. 2) follows immediately from $$(\lambda+A^\varepsilon)^{-1} \rightarrow (\lambda+A)^{-1} \quad \text{strongly in $L^2$}$$ and $(\lambda+A^\varepsilon)^{-1}(x,y) \leq C(\lambda-c\Delta)^{-1}(x,y)$ for generic constants $0<c,C<\infty$, an immediate consequence of \eqref{UGB_p}.

\medskip

(\textit{iii}) The  upper bound in \eqref{GB_u} of Theorem \ref{apr_GB} yields
$$
\|e^{-t\Lambda_1^\varepsilon}\|_{1 \rightarrow \infty} \leq c_2e^{t\omega_2}t^{-\frac{d}{2}}, \quad t>0, \quad \varepsilon>0
$$
with generic* constants $c_2$, $\omega_2<\infty$. Using Theorem \ref{nash_apost_thm}(\textit{ii}) and applying Fatou's lemma, we obtain
$
\|e^{-t\Lambda_1}\|_{1 \rightarrow \infty} \leq c_2e^{t\omega_2}t^{-\frac{d}{2}}$, $t>0.
$
Hence $e^{-t\Lambda_1}$ is an integral operator for every $t>0$.

\smallskip

(\textit{iv})  The a priori bounds \eqref{GB_u} of of Theorem \ref{apr_GB}, and  Theorem \ref{nash_apost_thm}(\textit{ii}), yield for every pair of bounded measurable subsets $S_1$, $S_2 \subset \mathbb R^d$:
$$
c_1 e^{t\omega_1} \langle \mathbf{1}_{S_1},e^{t\sigma_1\Delta}\mathbf{1}_{S_2} \rangle \leq \langle \mathbf{1}_{S_1},e^{-t\Lambda_1}\mathbf{1}_{S_2} \rangle \leq c_2 e^{t\omega_2} \langle \mathbf{1}_{S_1},e^{t\xi_1\Delta}\mathbf{1}_{S_2} \rangle.
$$
Since $e^{-t\Lambda_1}$ is an integral operator for every $t>0$, assertion (\textit{iv}) follows by applying the Lebesgue Differentiation Theorem.

\smallskip

(\textit{v}) For every $\varepsilon>0$, $\langle e^{-t\Lambda^\varepsilon}(x,\cdot)\rangle=1$, $x \in \mathbb R^d$. 
Fix $t>0$ and $\Omega \subset \mathbb R^d$, a bounded open set. By the upper bound \eqref{GB_u} of Theorem \ref{apr_GB}, for every $\gamma>0$ there exists $R=R(\gamma,t,\Omega)>0$ such that, for every  $x \in \Omega$,
$
\langle e^{-t\Lambda^\varepsilon}(x,\cdot)\mathbf{1}_{B^c(0,R)}(\cdot)\rangle<\gamma$, so
$
\langle e^{-t\Lambda^\varepsilon}(x,\cdot)\mathbf{1}_{B(0,R)}(\cdot)\rangle \geq 1 - \gamma.
$ Hence
$$
\langle \mathbf{1}_\Omega e^{-t\Lambda^\varepsilon}\mathbf{1}_{B(0,R)}\rangle \geq (1 - \gamma)|\Omega|. 
$$
Applying Theorem \ref{nash_apost_thm}(\textit{ii}), we obtain 
$$
\frac{1}{|\Omega|}\langle \mathbf{1}_\Omega e^{-t\Lambda}1\rangle \geq \frac{1}{|\Omega|}\langle \mathbf{1}_\Omega e^{-t\Lambda}\mathbf{1}_{B(0,R)}\rangle\geq 1 - \gamma.
$$
Applying the Lebesgue Differentiation Theorem, we obtain $\langle e^{-t\Lambda}(x,\cdot)\rangle \geq 1-\gamma$ for a.e.\,$x \in \mathbb R^d$.
In turn, the opposite inequality $\langle e^{-t\Lambda}(x,\cdot)\rangle \leq 1$ for a.e.\,$x \in \mathbb R^d$ follows easily using Theorem \ref{nash_apost_thm}(\textit{ii}), and hence $1\geq \langle e^{-t\Lambda}(x,\cdot)\rangle \geq 1-\gamma$. The proof of (\textit{v}) is completed.

\smallskip

(\textit{vi}) Put $u_\varepsilon(t,x):=e^{-t\Lambda^\varepsilon}f(x)$. Repeating the argument in \cite[sect.\,3]{FS} which appeals to the ideas of E.\,De Giorgi, we obtain assertion (\textit{vi}) for $u_\varepsilon$.
The result now follows upon applying Theorem \ref{nash_apost_thm}(\textit{ii}) and the Arzel\`{a}-Ascoli Theorem.

\smallskip

(\textit{vii}) follows from (\textit{iv}), (\textit{v}) and (\textit{vi}) using a standard argument for mollifiers.

\smallskip

(\textit{viii}) is proved repeating the argument in \cite[sect.\,2]{Da}.

\smallskip

(\textit{ix}) follows repeating the argument in \cite{Ou}.

\smallskip

(\textit{x}) In the proof of (\textit{i}) we obtain the resolvent representation as the K.\,Neumann  series
$$(\zeta+\Lambda_1)^{-1}=(\zeta+A_1)^{-1}(1+T_1)^{-1} \in \mathcal B(L^1), \quad \Real\zeta \geq \lambda_0,$$
where $\lambda_0=\lambda_0\big(n_e(b,h)\big)>0$, $T_1:=(b \cdot \nabla)_1(\zeta +A_1)^{-1} \in \mathcal{B}(L^1)$.
 The latter yields $\|\nabla (\zeta + \Lambda_1)^{-1}\|_{1 \rightarrow 1} \leq c (\Real \zeta)^{-\frac{1}{2}}$. Indeed, $\|\nabla(\zeta+A_1)^{-1}\|_{1\to 1}\leq c(\Real \zeta)^{-\frac{1}{2}}$ (integrating ($\star$) in $t \in [0,\infty[$ in the proof of Theorem \ref{grad_thm}), so the resolvent representation yields the required bound. The latter now easily yields the case $1/2<\alpha<1$.

\bigskip

\section{Proof of Theorem \ref{grad_thm}}

\label{grad_thm_sect}

It suffices to carry out the proof on $C^\infty_c$ for smooth bounded $a \in (H_{\sigma,\xi})$, $b$, and then apply Theorem \ref{nash_apost_thm}(\textit{ii})
 using the closedness of the gradient.

First, let $0<t\leq h$.

The Duhamel formula for $\nabla e^{-t\Lambda_1}$ yields: 
\begin{equation}
\label{duhamel}
\|\nabla e^{-t\Lambda_1}f\|_{1}  \leq \|\nabla e^{-tA_1}f\|_1+\int_0^t \|\nabla e^{-(t-\tau)A_1}\|_{1 \rightarrow 1} \|b \cdot \nabla e^{-\tau \Lambda_1}f\|_1 d\tau, \quad f\in C^\infty_c.
\end{equation}
We will need (proved below):
\[
 \|\nabla e^{-tA_1}\|_{1 \rightarrow 1} \leq C/\sqrt{t},\tag{$\star$}
\]
\[
\int_0^t\frac{C}{\sqrt{t-\tau}}\|b \cdot \nabla e^{-\tau \Lambda_1}f\|_1 d\tau  \leq C\sup_{x\in\mathbb{R}^d} \int_0^t \frac{1}{\sqrt{t-\tau}} \sqrt{ e^{\delta \tau \Delta} b_a^2(x)} \sqrt{\mathcal N^u_\delta(\tau,x)}d\tau\, \|f\|_1,\tag{$\star\star$}
\]
\[
\mathcal N^u_\delta(\tau,x) \leq \frac{C_2}{\tau}, \tag{$\star\star\star$}
\]
where $\mathcal N^u_\delta(\tau,x) :=\langle\nabla u(\tau,x,\cdot)\cdot\frac{a(\cdot)}{k_\delta(\tau,x,\cdot)}\cdot\nabla u(\tau,x,\cdot)\rangle$, $u(\tau,x,y)=e^{-\tau\Lambda}(x,y)$, $\delta>\xi$, the constants $C_1$, $C_2$, $\omega$ are generic.
We estimate the RHS of ($\star\star$): write $\int_0^t=\int_0^{t/2}+\int_{t/2}^t$ and use ($\star\star\star$) to obtain
\begin{align*}
\sup_{x\in\mathbb{R}^d} \int_0^{t/2} \frac{1}{\sqrt{t-\tau}} \sqrt{ e^{\delta \tau \Delta} b_a^2(x)} \sqrt{\mathcal N^u_\delta(\tau,x)}d\tau & \leq \frac{\sqrt{2 C_2}}{\sqrt{t}}\sup_{x\in\mathbb{R}^d}\int_0^{t/2} \sqrt{ e^{\delta \tau \Delta} b_a^2(x)} \frac{d\tau}{\sqrt{\tau}} \\
& \leq \frac{\sqrt{2 C_2}}{\sqrt{\delta t}} n_e(b,\frac{\delta h}{2}),
\end{align*}
\begin{align*}
\sup_{x\in\mathbb{R}^d} \int_{t/2}^t \frac{1}{\sqrt{t-\tau}} \sqrt{ e^{\delta \tau \Delta} b_a^2(x)} \sqrt{\mathcal N^u_\delta(\tau,x)}d\tau & \leq  \sqrt{C_2}\sup_{x\in\mathbb{R}^d}\int_{t/2}^t \frac{1}{\sqrt{t-\tau}}\sqrt{ e^{\delta \tau \Delta} b_a^2(x)} \frac{d\tau}{\sqrt{\tau}} \\
& (\text{we are using $e^{\delta \tau\Delta}b_a^2(x) \leq \frac{\xi d\beta }{8\delta }\frac{1}{\tau}+c(\beta)$ since $b \in \mathbf{F}$}) \\
& \leq \tilde{C}\int_{t/2}^t \frac{1}{\sqrt{t-\tau}}\frac{d\tau}{\tau} \leq \tilde{C}\frac{1}{\sqrt{t}}.
\end{align*}
Substituting ($\star$), ($\star\star$) and the last two estimates into \eqref{duhamel}, we have
$\|\nabla e^{-t\Lambda_1}\|_{1 \rightarrow 1} \leq \frac{c}{\sqrt{t}}$ for $0<t \leq h$. 

Also, for all $t > h$, $\|\nabla e^{-t\Lambda_1}\|_{1 \rightarrow 1} \leq \|\nabla e^{-h\Lambda_1}\|_{1 \rightarrow 1}\|e^{-(t-h)\Lambda_1}\|_{1 \rightarrow 1} \leq \frac{\tilde{c}}{\sqrt{h}}e^{(t-h)\omega_2}$ (cf.\,Theorem \ref{nash_apost_thm}). The latter yields the assertion of Theorem \ref{grad_thm} for all $t>0$.

\medskip

It remains to prove ($\star$)-($\star\star\star$).

Proof of ($\star$): We have for $\mathsf{h} \in \mathbb R^d$, $\mathsf{h}=(0,\dots,1,\dots,0)$ ($1$ is in the $i$-th coordinate, $1 \leq i \leq d$)
\begin{align*}
\|\mathsf{h} \cdot \nabla e^{-tA_1}f\|_1 & \leq \sup_{x \in \mathbb R^d}\sqrt{\langle k_\delta(t,x,\cdot) (\mathsf{h} \cdot a^{-1}(\cdot) \cdot \mathsf{h})\rangle}\sqrt{\mathcal N_\delta(t,x)}\|f\|_1  \\
&  \leq \sigma^{-\frac{1}{2}}\sup_{x \in \mathbb R^d}\sqrt{ \mathcal N_\delta(t,x)}\|f\|_1  = \sigma^{-\frac{1}{2}}\sqrt{\sup_{x \in \mathbb R^d} \mathcal N_\delta(t,x)}\|f\|_1,
\end{align*}
and so by Proposition \ref{TimeIndCorollary2}
$$
\|\nabla e^{-tA_1}f\|_1 \leq \frac{d \sqrt{\sigma^{-1}c_0}}{\sqrt{t}}\|f\|_1.
$$

The estimate ($\star\star$) follows using quadratic inequality.
 
Thus, we are left to prove $(\star\star\star)$. Integrating by parts, using the equation for $u(t,x,y)$ and $(\rm {UGB}^u), ({\rm UGB}^{\partial_tu})$ (see Theorem \ref{nash_apost_thm}\textit{(iv),(viii)}), we obtain for $0<t\leq h$ (below $c$ is a generic constant)
\[
\mathcal{N}_\delta^u(t,x)=\langle\nabla u\cdot\frac{a}{k_\delta}\cdot\nabla u\rangle=-\langle k^{-1}_\delta u\partial_t u\rangle-\langle k^{-1}_\delta ub\cdot\nabla u\rangle+\langle uk_\delta^{-2}\nabla k_\delta\cdot a\cdot\nabla u\rangle,
\]
\[
|\langle k^{-1}_\delta u\partial_t u\rangle|\leq \frac{c}{t}, \quad |\langle uk_\delta^{-2}\nabla k_\delta\cdot a\cdot\nabla u\rangle|\leq c|\langle\nabla k_\delta\cdot\frac{a}{k_\delta}\cdot \nabla u\rangle|.
\]
Clearly,
\[
|\langle \nabla k_\delta\cdot \frac{a}{k_\delta}\cdot\nabla u\rangle|\leq \frac{c}{\sqrt{t}}\sqrt{\mathcal{N}_\delta^u(t,x)}.
\]
\[
|\langle k^{-1}_\delta ub\cdot\nabla u\rangle|\leq c\sqrt{e^{\delta t\Delta}b_a^2(x)}\sqrt{\mathcal{N}_\delta^u(t,x)}\leq \hat{c}\frac{1}{\sqrt{t}}\sqrt{\mathcal{N}_\delta^u(t,x)}
\]
(due to $e^{\delta t\Delta}b_a^2(x) \leq \frac{\xi d\beta }{8\delta }\frac{1}{t}+c(\beta)$, see above). Now $(\star\star\star)$ is evident. 
 
The proof of Theorem \ref{grad_thm} is completed.

\bigskip

\section{Comments}

\label{sect3}

\textbf{1.~}The following result was proved in \cite{KiS} (the reader can compare it with Theorem \ref{nash_apost_thm}). It 
establishes quantitative dependence of the regularity properties of solutions to $(\partial_t + \Lambda)u=0$ with $b \in \mathbf{F}_\delta(A)$ on the value of $\delta$.

\begin{theorem}
\label{thm_fb}
Let $d \geq 3$. Assume that $b \in \mathbf{F}_\delta(A)$ for some $0 < \delta < 4.$ Set $r_c := \frac{2}{2-\sqrt{\delta}}$ and  $b_a^2:=b \cdot a^{-1} \cdot b \in L^2_{\loc}$. The following is true:

\smallskip

{\rm (\textit{i})} Let $\mathbf 1_n$ denote the indicator of $\{x \in \mathbb R^d \mid  \; b_a(x) \leq n \}$ and set $b_n := \mathbf 1_n b.$ Then the limit
\[
s\mbox{-}L^r\mbox{-}\lim_{n\rightarrow \infty} e^{-t \Lambda_r(a,b_n)},  \quad r \in I_c^o := ]r_c, \infty [,
\]
where $\Lambda_r(a,b_n):=A_r + b_n \cdot \nabla$, exists locally uniformly in $t \geq 0$ and determines a
positivity preserving, $L^\infty$ contraction, quasi contraction $C_0$ semigroup on $L^r$, say,
$e^{-t \Lambda_r(a,b)}$. 

\smallskip

{\rm (\textit{ii})} Define
\[
e^{-t \Lambda_{r_c}(a,b)} := \big[e^{-t \Lambda_r(a,b)} \upharpoonright L^1 \cap L^r \big]^{\rm clos}_{L^{r_c} \rightarrow L^{r_c} }, \quad \quad  r \in I_c^o.
\]
Then $e^{-t \Lambda_{r_c}(a,b)}$ is a $C_0$ semigroup and
\[
\|e^{-t \Lambda_r(a,b)} \|_{r \rightarrow r} \leq e^{t \omega_r}, \quad \omega_r = \frac{\lambda \delta}{2(r-1)},\quad r \in I_c:=[r_c,\infty[.
\]

\smallskip

{\rm (\textit{iii})} The interval $I_c$ is the maximal interval of quasi contractive solvability.

\smallskip

{\rm (\textit{iv})} For each $r \in I_c^o, \; e^{-t\Lambda_r(a,b)}$ is a holomorphic semigroup of quasi contractions in the sector
\[
|\arg t | \leq \frac{\pi}{2}-\theta_r, \quad 0 < \theta_r < \frac{\pi}{2}, \; \tan \theta_r \leq \mathcal K(2- r^\prime \sqrt{\delta})^{-1}, 
\]
where $\mathcal K = \frac{|r-2|}{\sqrt{r-1}} +r^\prime \sqrt{\delta}$ if $r\leq 2r_c$ and $\mathcal K=\frac{r-2 +r\sqrt{\delta}}{\sqrt{r-1}}$ if  $ r>2r_c.$

\smallskip

{\rm (\textit{v})} $e^{-t \Lambda_r(a,b)}$, $r \in I_c$, extends to a positivity preserving, $L^\infty$ contraction, quasi bounded holomorphic semigroup on $L^r$ for every $r \in I_m:= ]\frac{2}{2- \frac{d-2}{d}\sqrt{\delta}}, \infty[$. 

\smallskip

{\rm (\textit{vi})} The interval $I_m$ is the maximal interval of quasi bounded solvability.

\smallskip

{\rm (\textit{vii})} For every $r \in I_m$ and $q>r$ there exist constants $c_i = c_i(\delta, r, q)$, $i=1,2$ such that the $(L^r, L^q)$ estimate
\begin{equation*}
\|e^{-t \Lambda_r(a,b)} \|_{r \rightarrow q} \leq c_1 e^{c_2 t} \; t^{-\frac{d}{2} (\frac{1}{r}-\frac{1}{q})} 
\end{equation*}
is valid for all $t > 0.$

\smallskip

{\rm (\textit{viii})} Let $\delta < 1$, and let $a_n \in (H_{\sigma,\xi})$,  $b_n:\mathbb R^d \rightarrow \mathbb R^d$, $n=1,2,\dots$ be smooth and such that
\[
a_n \rightarrow a \text{ strongly in }[L^2_\loc]^{d\times d}, \quad b_n \rightarrow b \text{ strongly in } [L^2_\loc]^d
\]
and 
$b_n \in \mathbf F_\delta(A^n)$ with $c(\delta)$ independent of $n$, where $A^n \equiv -\nabla \cdot a_n \cdot \nabla$.
Then $$e^{-t \Lambda_r(a, b)} = s \mbox{-} L^r \mbox{-} \lim_{n \uparrow \infty} e^{-t \Lambda_r(a_n, b_n)}$$ whenever $r  \in I^o_c$, where $\Lambda_r(a_n,b_n)=-\nabla \cdot a_n \cdot \nabla + b_n\cdot\nabla$ of domain $W^{2,r}$.
\end{theorem}

\begin{remarks}
(a) For $\delta<1$, the corresponding to $\Lambda$ quadratic form $t[u]=\langle a \cdot \nabla u,\nabla u\rangle + \langle b \cdot \nabla u,u\rangle$, $D(t)=W^{1,2}$ possesses the Sobolev embedding property $$\Real t[u] \geq c_S\|u\|^2_{2j}, \quad j=\frac{d}{d-2}.$$ \textit{This ceases to be true already for $\delta=1$. 
The same occurs for $1<\delta<4$ and $r=r_c$.}

\smallskip

(b) The intervals $I_c$, $I_m$ are maximal already for $a=I$ and $b(x)=\sqrt{\delta}\frac{d-2}{2}|x|^{-2}x$.

\smallskip

(c) Assertions (\textit{i})-(\textit{iv}) are in fact valid for symmetric $a \in [L^1_{\loc}]^{d \times d}$ such that $a \geq \sigma I$, $\sigma>0$, and $b_a^2 \in L^1 + L^\infty$, see \cite[Theorem 4.2]{KiS}.

\smallskip

(d) While for $b \in \mathbf{F}_\delta(A)$, $\delta<1$
one first constructs the semigroup in $L^2$ (using the method of quadratic forms)
and then proves the corresponding convergence results, in the case $b \in \mathbf{F}_\delta(A)$, $1 \leq \delta<4$ the convergence result of Theorem \ref{thm_fb}(\textit{i})  becomes the means of construction of the semigroup.
\end{remarks}

\medskip

\textbf{2.~}Note that $\mathbf{N}_e\cap\mathbf{F}\subset \mathbf{K}^d\subset\mathbf{F}$, where $\mathbf{F}:=\cup_{\beta>0}\mathbf F_\beta(-\Delta)$,  and 
 \[
 \mathbf{K}^d:=\{|b|\in L^2_\loc\mid \kappa_d(b,h):=\sup_{x\in\mathbb{R}^d}\int_0^h e^{t\Delta}|b|^2(x)dt <\infty \text{ for some } h>0\}.
 \]
Indeed, using $b \in \mathbf{F}$, we have $e^{t\Delta}b^2(x)\equiv\langle k(t,x,\cdot)b^2(\cdot)\rangle\leq \beta\|\nabla\sqrt{k(t,x,\cdot)}\|_2^2+c(\beta)=\frac{\beta d}{8}\frac{1}{t}+c(\beta)$ for some $\beta>0$ and $c(\beta)$. Therefore, for $0<t\leq h$,
$$
e^{t\Delta}b^2(x) \leq \sqrt{\frac{\beta d}{8}+c(\beta)h}\,\sqrt{e^{t\Delta}b^2(x)}\frac{1}{\sqrt{t}},
$$
and so the condition $b \in \mathbf{N}_e$ now yields the required. In turn, the inclusion $\mathbf{K}^d\subset\mathbf{F}$ is well known (use the fact that $b \in \mathbf{K}^d$ is equivalent to $\||b|^2(\lambda-\Delta)^{-1}\|_{1 \rightarrow 1}<\infty$, $\lambda>0$).

\smallskip

\textbf{3.~}Let us fix a continuous function $\phi:[0,\infty[ \rightarrow [0,\infty[$ satisfying the following properties: 

1) $\phi(0)=0$, 

2) $\phi(t)/t\in L^1[0,1]$. 

Put
$$n_\phi(b,h)=\sup_{x\in \mathbb{R}^d}\int_0^h e^{t\Delta}b^2(x)\frac{dt}{\phi(t)}.$$
If $n_\phi(b,h)<\infty$ for some $h>0$, then we write $b \in \mathbf{N}_\phi$.

The class $\mathbf{N}_\phi$ arises as the class providing the two-sided Gaussian on the heat kernel of  $-\nabla \cdot a(t,x) \cdot \nabla + b(t,x) \cdot \nabla$, where $a(t,x)$ is a measurable uniformly elliptic matrix, see \cite{S2}, \cite{LS}.
Since (for $b=b(x)$)
$$
\int_0^h \sqrt{e^{t\Delta}b^2(x)}\frac{dt}{\sqrt{t}} \leq \biggl[\int_0^h e^{t\Delta}b^2(x)\frac{dt}{\phi(t)}\biggr]^{\frac{1}{2}}\biggl[\int_0^h \frac{\phi(t)}{t}dt\biggr]^{\frac{1}{2}},
$$
we have
$
\mathbf{N}_\phi  \subset \mathbf{N}_e$ for every admissible $\phi$.
Moreover, since $\phi$ is continuous and $\phi(0)=0$, it is seen that $n_\phi(b,h) > k_d(b,h)$, and so $\mathbf{N}_\phi \subset \mathbf{K}^{d}$.
Thus,
$$
\mathbf{N}_\phi \; \subset \; \mathbf{N}_e \cap \mathbf{K}^d \; \subset \; \mathbf{K}^{d+1} \cap \mathbf{K}^d.
$$
The need for more restrictive assumption ``$b \in \mathbf{N}_\phi$'' when $a=a(t,x)$ is dictated by the subject matter: in the time-dependent case there are no estimates $\mathcal{N}(t)$, $\hat{\mathcal{N}}(t) \leq c(t)$ for any $c(t)$, cf.\,the previous comment.

\smallskip

\textbf{4.~}Let us comment more on  classes $\mathbf{K}^{d+1}$ and $\mathbf{F}$. 

Note that $\mathbf{K}^{d+1} \not\subset \mathbf{F}$: There are $b \in \mathbf{K}^{d+1}$ such that, for a given $p>1$, $|b| \not\in L^{p}_{\loc}$,  e.g. consider 
$$
|b(x)|=\mathbf{1}_{B(0, 1)}(x)|x_1|^{-\alpha_p}, \quad 0<\alpha_p<1.
$$
On the other hand, already $[L^d]^d \not\subset \mathbf{K}^{d+1}$, and so $\mathbf{F} \not \subset \mathbf{K}^{d+1}$. [Indeed, let
\[
|b(x)|=\mathbf{1}_{B(0,e^{-1})}(x)|x|^{-1}|\log |x||^{-\alpha},\;\alpha>d^{-1}, \;d\geq 3.
\]
 Then $\|b\|_d<\infty$ and $k_{d+1}(b,h)=\infty$.]

This dichotomy between the classes $\mathbf{K}^{d+1}$ and $\mathbf{F}$ was resolved in \cite{Ki,KiS} with development of the Sobolev regularity theory of $-\Delta + b \cdot \nabla$ for $b$ in the class $$\mathbf{F}^{\scriptscriptstyle 1/2}=\big\{b \in L^1_{\loc} \mid \lim_{\lambda \rightarrow \infty}\||b|^{\frac{1}{2}}(\lambda-\Delta)^{-\frac{1}{4}}\|_{2 \rightarrow 2} <\infty \big\}$$ (introduced in \cite{S} as the class responsible for the $(L^p,L^q)$ estimate on the semigroup) that contains $\mathbf{K}^{d+1} + \mathbf{F}:=\{b_1+b_2 \mid b_1 \in \mathbf{K}^{d+1}, b_2 \in \mathbf{F}\}$.

By analogy, one can ask if it is possible to extend
the convergence results in Theorem \ref{nash_apost_thm} and Theorem \ref{thm_fb}, or ($L^p,L^q$) estimates,  to $-\nabla \cdot a \cdot \nabla + b \cdot \nabla$ with a measurable $a \in (H_{\sigma,\xi})$ and $b=b_1+b_2$ with $b_1 \in \mathbf{N}_e$, $b_2 \in \mathbf{F}_{\delta}(A)$.

\medskip

\textbf{5.~}Theorem \ref{nash_apost_thm}(\textit{iv}), (\textit{viii}) (the two-sided Gaussian bounds on the heat kernel and its time derivative) can be extended to more general operator
$$
\Lambda(a,b,\hat{b})=-\nabla \cdot a \cdot \nabla + b \cdot \nabla + \nabla \cdot \hat{b}
$$
with $a \in (H_{\sigma,\xi})$, and $(b,\hat{b} \in \mathbf{N}_e, \hat{b} \in \mathbf F)$ or ($b,\hat{b} \in \mathbf{N}_e, b \in \mathbf F$),
provided that $n(b,h)$, $n(\hat{b},h)$ are sufficiently small. Note that the above assumptions on $b$ and $\hat{b}$ are non-symmetric, i.e.\,the presence of $b \in \mathbf{N}_e$ forces $\hat{b}$ to be more regular: $\hat{b} \in \mathbf{N}_e\cap \mathbf F$, and vice versa. We also note that here the form-boundedness assumption seems to be justified. 
The proof follows the argument in the present paper but with the Nash's functions $\mathcal N$, $\hat{\mathcal N}$ defined with respect to $u(t,x,y):=e^{-t\Lambda(a,b)}(x,y)$. We will address this matter in detail elsewhere.

\medskip

\textbf{6.~}The authors do not know if there is a proof of the Harnack inequality for $\Lambda = -\nabla \cdot a \cdot \nabla + b \cdot \nabla$, $a \in (H_{\sigma,\xi})$, $b \in \mathbf{N}_e$ that does not use the lower bound on $e^{-t\Lambda}(x,y)$.

\end{document}